\documentclass[10pt,arxiv=true,scale=.8036]{agn_article}
\usepackage[lastauthorbreak=true]{agn_all}
\usepackage[outdir=./images/]{epstopdf}
\graphicspath{./images/}
\usepackage{subcaption}
\usepackage[font=small]{caption}
\newcommand{\sample}{343}

\providecommand{\meas}[1]{\abs{#1}}
\providecommand{\W}[2]{\ensuremath{W^{#1,#2}}}
\providecommand{\WBiot}{W_{\subtext{Biot}}}

\providecommand{\distspace}{\dist}
\providecommand{\casesif}{\;\textnormal{if }\;}
\providecommand{\casesotherwise}[1][]{\;\textnormal{otherwise#1}}
\providecommand{\quoteref}[1]{\enquote{\emph{#1}}}

\providecommand{\Wdist}{W_{\subtext{dist}}}
\providecommand{\Lcal}{\mathcal{L}}
\renewcommand{\grad}{{\operatorfont D}\.}
\providecommand{\hide}[1]{}
\tikzstyle{plotstyle}=[samples=\sample, color=black, very thick]
\tikzstyle{biotstyle}=[samples=\sample, color=black, very thick, dashed]
\tikzstyle{diststyle}=[samples=\sample, color=R, thick]
\tikzstyle{Qdiststyle}=[samples=\sample, color=B, ultra thick]
\title{Quasiconvex relaxation of planar Biot-type energies and the role of determinant constraints}
\defineaffiliation{rub}{%
	Chair of Continuum Mechanics,
	Ruhr University Bochum,
	Universitätsstraße 150,
	44801 Bochum,
	Germany%
}
\defineaffiliation{rubhead}{%
	Head of Chair of Continuum Mechanics,
	Ruhr University Bochum,
	Universitätsstraße 150,
	44801 Bochum,
	Germany%
}
\defineaffiliation{tud}{%
	Faculty of Mathematics,
	Technical University of Dresden,
	Zellescher Weg 12--14,
	01069 Dresden, Germany
}
\defineauthor{koehler}{Maximilian~K\"ohler}{rub}{maximilian.koehler@rub.de}
\defineauthor{balzani}{Daniel~Balzani}{rub}{}
\defineauthor{sander}{Oliver~Sander}{tud}{oliver.sander@tu-dresden.de}
\defineauthor{nebel}{Lisa~Julia~Nebel}{tud}{}
\knownauthors[martin]{martin,ghiba,koehler,balzani,sander,neff}
\begin{document}
\maketitle
\begin{abstract}
	\noindent We derive the quasiconvex relaxation of the Biot-type energy density $\norm{\sqrt{\grad\varphi^T \grad\varphi}-\id_2}^2$ for planar mappings $\varphi\colon\R^2\to \R^2$ in two different scenarios. First, we consider the case $\grad\varphi\in\GLp(2)$, in which the energy can be expressed as the squared Euclidean distance $\dist^2(\grad\varphi,\SO(2))$ to the special orthogonal group $\SO(2)$. We then allow for planar mappings with arbitrary $\grad\varphi\in\R^{2\times 2}$; in the context of solid mechanics, this lack of determinant constraints on the deformation gradient would allow for self-interpenetration of matter. We demonstrate that the two resulting relaxations do not coincide and compare the analytical findings to numerical results for different relaxation approaches, including a rank-one sequential lamination algorithm, trust-region FEM calculations of representative microstructures and physics-informed neural networks.
\end{abstract}
\vspace*{.49em}%
\vspace*{.49em}%
\tableofcontents

\section{Introduction}

In the context of nonlinear hyperelasticity, we consider the deformation $\varphi\col\Omega\to\R^n$ of an elastic body $\Omega\subset\R^n$ and the basic minimization problem
\begin{equation}\label{eq:basic_minimization_problem}
	I(\varphi) = \int_\Omega W(\grad\varphi(x))\,\dx \;\to\;\min_{\varphi\in\adm}
	\,.
\end{equation}
Here, $\Omega\subset\R^n$ is a connected domain and $\adm\subset \W1p(\Omega;\R^n)$ is the set of admissible functions, which is usually given by specific boundary conditions and includes the \emph{determinant constraint} $\det\grad\varphi(x)>0$ for a.e.\ $x\in\Omega$, which ensures that there is no local self-interpenetration of the solid body.\footnote{For incompressible solids, the stronger constraint $\det\grad\varphi=1$ can be imposed as well.} In particular, the \emph{elastic energy potential} $W\col\GLpn\to\R$, which provides the model of the elastic material, is naturally defined on the set $\GLpn$ of matrices with positive determinant.

In order to apply the direct methods in the calculus of variations to ensure the existence of solutions to \eqref{eq:basic_minimization_problem}, the energy $W$ needs to satisfy certain generalized convexity properties. In particular, the weak lower semicontinuity of the energy functional $I$ requires $W$ to be \emph{quasiconvex} \cite{morrey1952quasi}: if $W$ is not quasiconvex, energy-infimizing sequences of deformations can exhibit successively finer \emph{microstructures}. In this case, the minimal (or infimal) elastic energy under given Dirichlet boundary conditions can be determined via the \emph{relaxation} or \emph{quasiconvex envelope} of $W$.

In the following, we will consider a specific energy potential -- the \emph{Biot-type energy} $\WBiot$ -- and discuss the role played by the determinant constraint for the relaxation. More specifically, we will demonstrate both analytically and numerically that enforcing the condition $\det F>0$ on the deformation gradient $F=\grad\varphi$ is indeed necessary in order to find the correct results for $\WBiot$, whereas the relaxed energy potential for the related function $\Wdist$ does not require the application of constraints.

\subsection{Isotropic energy potentials and the Seth-Hill strain family}
An elastic energy potential $W\col\GLp(n)\to\R$ or $W\colon\R^{n\times n}\to\R$ is often assumed to be \emph{objective} (or \emph{frame-indifferent}) as well as \emph{isotropic}, i.e.\ to satisfy
\begin{equation}
	W(Q_1F\,Q_2) = W(F)
	\qquad\text{for all }\;F\in\GLp(n)
	\quad\text{and all }\;\; Q_1,Q_2\in\SO(n)\,,
\end{equation}
where $\GLp(n)=\{X\in\R^{n\times n}\setvert\det X> 0\}$ denotes the group of $n\times n$-matrices with positive determinant $\SO(n)=\{X\in \Rnn \setvert Q^T Q=\id_n\,,\;\det Q=1\}$ is the special orthogonal group; here and throughout, $\id_n$ denotes the $n\times n$ identity tensor.
It is well known that any objective and isotropic function $F\mapsto W(F)$ on $\GLpn$ can be expressed in terms of the singular values\footnote{For any $F\in\Rnn$, the singular values of $F$ are the eigenvalues of $\sqrt{F^TF}$.} of the argument, i.e.
there exists a unique symmetric function $g\col\R_+^n\to \R$ such that $W(F)=g(\lambda_1,\dotsc,\lambda_n)$ for all $F\in\GLpn$
with singular values $\lambda_1,\dotsc,\lambda_n$.

If $f\colon[0,\infty)\to\R$ is injective, then any objective and isotropic $W\col\GLp(n)\to\R$ can also be expressed as
\begin{align}
	W(F)=\widetilde{g}(f(\lambda_1),f(\lambda_2),...,f(\lambda_n))
\end{align}
with a symmetric function $\widetilde{g}\col\R\to \R$. In particular, this representation is possible for functions $f=f_{(m)}$ of the form
\begin{align}
	f_{(m)}(x)=
	\begin{cases}
		\frac{1}{2\.m}(x^{2m}-1) &\casesif m\neq 0\,, \\
		\log x &\casesif m=0\,,
	\end{cases}
\end{align}
which correspond to the commonly used \cite{boehlkeBertram2002} \emph{strain tensors} of \emph{Seth-Hill type} \cite{seth1961generalized,hill1970},
\begin{align}
	E_{(m)}=
	\begin{cases}
		\frac{1}{2\.m}(U^{2\.m}-\id_n) &\casesif m\neq 0\,,\\
		\log U &\casesif m=0\,;
	\end{cases}
\end{align}
here, $\log U$ is the principal matrix logarithm of the stretch tensor $U=\sqrt{F^TF}$ and the identity tensor on $\R^{n \times n}$ is denoted by $\id_n$.
In general, a (material) strain tensor is commonly defined as a \quoteref{uniquely invertible isotropic second order tensor function} of the right Cauchy-Green deformation tensor $C=F^TF$ \cite[p.~268]{truesdell60}.\footnote{Additional properties such as monotonicity are sometimes required of strain tensors, see for example \cite[p.\ 230]{Hill68} or \cite[p.~118]{Ogden83}.} %
Due to the invertibility of strain tensor mappings, any energy function $W$ can also be written in the form
\begin{equation}
	W(F) = \widetilde{W}(E_{(m)}(F))
\end{equation}
for any strain tensor $E_{(m)}$.

In contrast to a strain tensor, a \emph{strain measure} is an arbitrary mapping $\omega\colon\GLp(n)\to\R$ such that $\omega(F)=0$ if and only if $F\in\SOn$. Examples of strain measures include the \emph{Frobenius norms} of the Seth-Hill strain tensors
\begin{equation}
	\omega_{(m)}(F) = \norm{E_{(m)}}=\sqrt{\sum_{i=1}^n f_{(m)}^2(\lambda_i)}\,.
\end{equation} 
From this perspective, a strain measure indicates how much a deformation gradient $F\in\GLp(n)$ or $F\in\R^{n\times n}$ differs from a pure rotation \cite{agn_neff2014grioli,agn_neff2015geometry}.

A typical starting point in nonlinear elasticity is the isotropic quadratic energy in the Green strain tensor $E_{(1)}=\frac{1}{2}(F^TF-\id_3)$, i.e.\ the St.~Venant--Kirchhoff type strain energy
\[
	\WSVK\colon\R^{2\times2}\to\R\,,\quad\;\;\WSVK(F)= W_{(1)}(F)=\norm{{F^T F}-\id}^2=\norm{C-\id}^2\,,
\]
which represents a straightforward extension of linear elasticity and allows for easy algebraic computation of the various stress-tensors, since the relation between the (second Piola-Kirchhoff) stress and the Green strain is linear. Furthermore, a nonlinear elastic material model can be defined analogously via the energy
\[
	W_{(m)}(F) = \omega_{(m)}^2(F)=\norm{E_{(m)}}^2
\]
for any strain tensor $E_{(m)}$, such as the lateral-contraction-free Hencky strain energy \cite{Hencky1929,agn_neff2014axiomatic}, which is given by $\WH(F)=W_{(0)}(F)=\norm{\log U}^2$. In the following, we will focus on the Biot-type energy
\[
	\WBiot\colon \R^{2\times 2}\to \R\,, \quad\;\; \WBiot(F)= W_{(1/2)}(F)=\norm{\sqrt{F^T F}-\id_2}^2=\norm{\sqrt{C}-\id_2}^2
	\,.
\]

While all energy functions $W_{(m)}$ can be used as a simple extension of linear elasticity, they generally lack a number of constitutive properties which are considered beneficial (or even necessary) for nonlinear elastic models suitable for describing large strains. In particular, these energies (including $\WBiot=W_{(1/2)}$) are not \emph{rank-one convex} \cite{boehlkeBertram2002,Silhavy05} (cf.~\cite{Raoult1986,Neff_Diss00,Bruhns01} for earlier results on $m=1$ and $m=0$) and therefore not quasiconvex \cite{morrey1952quasi}.

\subsection{Generalized convexity and relaxation}

It is important to note that for generalized notions of convexity (convexity, polyconvexity, quasiconvexity, rank-one convexity) -- and especially for the corresponding convex envelopes -- the domain of the function is of crucial importance. In particular, it needs to be carefully distinguished whether generalized convex envelopes are computed with respect to the entire matrix space $\Rnn$ or the group $\GLpn$ of matrices with positive determinant.
\begin{definition}
	\label{definition:convexities}
	Let $n\in\N$ and $p\in[1,\infty]$.
	\begin{itemize}
		\item[1)]
		A function $W\col\Rnn\to\R\cup\{+\infty\}$ is called
		\begin{itemize}
			\item[i)]
			\emph{rank-one convex} if
			for all $F_1,F_2\in\Rmn$ with $\rank(F_2-F_1)=1$,
			\[
			W((1-t)F_1+tF_2) \leq (1-t)\,W(F_1) + t\,W(F_2) \qquad\tforall t\in[0,1]\,;
			\]
			%
			%
			\item[ii)]
			\emph{\W1p-quasiconvex} \cite{Ball84b} if
			for every bounded open set $\Omega\subset\R^n$ with $\meas{\partial\Omega}=0$,
			\begin{equation}\label{eq:quasionvexityDefinition}
			\int_{\Omega} W(F+\grad\vartheta(x)) \,\dx \geq \meas{\Omega}\cdot W(F)
			\end{equation}
			for all $F\in\Rnn$ and all $\vartheta\in\W1p_0(\Omega;\R^n)$ for which the integral in \eqref{eq:quasionvexityDefinition} exists;
			\item[iii)]
			\emph{quasiconvex} if $W$ is \W1\infty-quasiconvex.
		\end{itemize}
		\item[2)]
		A function $W\col\GLpn\to\R$ is called rank-one convex [\W1p-quasi\-con\-vex/quasi\-convex] if the function
		\[
		\What\col\Rnn\to\R\cup\{+\infty\}\,,\quad
		\What(F)=
		\begin{cases}
		W(F) &\casesif F\in\GLpn\,,\\
		+\infty &\casesif F\notin\GLpn\,,
		\end{cases}
		\]
		is rank-one convex [\W1p-quasiconvex/quasiconvex].
		\item[3)]
		A function $W\col\GLpn\to\R$ is called convex if there exists a convex function $\What\col\Rnn\to\R$ such that $\What(F)=W(F)$ for all $F\in\GLpn$.
		\item[4)]
		The \emph{rank-one convex envelope} $RW$, the \emph{quasiconvex envelope} $QW$ and the \emph{convex envelope} $CW$ of a function $W\col\Rnn\to\R\cup\{+\infty\}$ are given by
		\begin{alignat*}{2}
			RW(F) &= \sup \{ w(F) \setvert w\col\Rnn\to\R\cup\{+\infty\} \,\text{ rank-one convex, }\;\;&& w\leq W \}\,,\\
			QW(F) &= \sup \{ w(F) \setvert w\col\Rnn\to\R\cup\{+\infty\} \,\text{ quasiconvex, }&& w\leq W \}\,,\\
			CW(F) &= \sup \{ w(F) \setvert w\col\Rnn\to\R\cup\{+\infty\} \,\text{ convex, }&& w\leq W \}\,,
		\end{alignat*}
		respectively.
	\end{itemize}
\end{definition}


By its original definition, quasiconvexity
ensures stability of homogeneous deformations against \emph{any} interior perturbation. Moreover, according to the \emph{Dacorogna formula} \cite{dacorogna1982quasiconvexity}, the quasiconvex envelope of a function $W\col\Rnn\to\R$ is given by its \emph{relaxation}, i.e.\ $QW(F)$ is the energy infimum over all deformations with homogeneous Dirichlet boundary conditions induced by $F$:
\begin{equation}\label{eq:dacorogna_formula}
	QW(F) = \frac{1}{\abs{\Omega'}}\,\inf_\vartheta \int_{\Omega'} W(F+\grad\vartheta(x)) \,\dx\,,
	\qquad
	\vartheta\in W_0^{1,\infty}(\Omega';\R^n)
	\,,
\end{equation}
where $\Omega'\subset\R^n$ is an arbitrary domain. In nonlinear elasticity, however, the restriction of $W$ to $\GLpn$ naturally leads to the constraint $\det(F+\grad\vartheta)>0$ a.e.\ in the relaxation formula \eqref{eq:dacorogna_formula}. In the following, we will investigate the difference between these two notions of relaxation more closely for the Biot energy $\WBiot$.

On the other hand, rank-one convexity, which for a $C^2$-function $W$, is equivalent to the \emph{Legendre-Hadamard (LH) ellipticity} condition
\[
	D^2_F W(F)(\xi\otimes\eta,\xi\otimes\eta)\geq0 \quad\text{ for all }\;\xi,\eta\in\mathbb{R}^n\setminus \{0\},\;\; F\in \GLpn
	\,,
\]
is often considered as a minimal material stability condition, since (in addition to real wave speeds) LH-ellipticity implies the stability of homogeneous states against \emph{infinitesimal} interior material perturbations.

The generalized convex envelopes of a given function are, in general, challenging to determine. In a landmark paper \cite{le1995quasiconvex}, Le Dret and Raoult explicitly computed the quasiconvex envelopes of St.~Venant-Kirchhoff-type energies on  $\R^{m\times n}$, including the relaxation
\begin{align}
	Q\WSVK(F)&\colonequals\sup \{Z\in \R^{2\times 2}\to \R \setvert Z \;\,\text{quasiconvex}\,,\; Z\leq \WSVK\}
\end{align}
of the classical St.~Venant-Kirchhoff energy on $\R^{2\times2}$. In the planar case, relaxation formulas have also been given for more general classes of nonlinear energies on $\GLp(2)$, including \emph{isochoric} (or \emph{conformally invariant}) energies \cite{agn_martin2015rank,agn_martin2019envelope}, and on the special linear group $\SL(2)$, i.e.\ incompressible energies \cite{agn_ghiba2018rank,agn_martin2019quasiconvex}.

%

In the following, we will provide the quasiconvex envelope of the Biot-type energy
\begin{align}
	\WBiot\colon \R^{2\times 2}\to \R, \qquad \WBiot(F)=\norm{\sqrt{F^TF}-\id_2}^2
\end{align}
both on $\R^{2\times 2}$ and on $\GLp(2)$. The relaxation results are summarized in Section~\ref{section:conclusion}.

\section{The Euclidean distance to \boldmath$\SO(2)$ and Biot-type energy \boldmath$\WBiot$}

As indicated above, the Biot strain measure $\WBiot(F)=\norm{\sqrt{F^TF}-\id_2}^2$ is directly related to the Euclidean distance of the deformation gradient $F$ to the special orthogonal group $\SO(2)$. More specifically,\footnote{Note that $\min\limits_{R\in\SO(2)}\norm{F-R}^2=\min\limits_{R\in\SO(2)}\norm{R^TF-\id_2}^2$ due to the unitary invariance of the Frobenius matrix norm.}
\begin{equation}\label{eq:biotDistRelation}
	\WBiot(F)=\norm{\sqrt{F^TF}-\id_2}^2
	= \dist^2(F,\SO(2))
	\colonequals \min_{R\in\SO(2)}\norm{F-R}^2
	\qquad\text{for all }\; F\in\GLp(2)
	\,.
\end{equation}
However, it is worth noting that eq.~\eqref{eq:biotDistRelation} is indeed only valid under the condition $F\in\GLp(2)$, which is generally assumed to hold for the deformation gradient $F$ in nonlinear elasticity.
For example, with $X=\diag(1,-1)\notin\GLp(2)$ and any $R\in\SO(2)$, we find
\begin{align*}
	\norm{X-R}^2
	&= (1-R_{11})^2 + (-1-R_{22})^2 + R_{12}^2 + R_{21}^2
	\\&= (1-\cos(\alpha))^2 + (1+\cos(\alpha))^2 + 2\.\sin^2(\alpha)
	= 2 + 2\.(\cos^2(\alpha)+\sin^2(\alpha))
	= 4
\end{align*}
for some $\alpha\in(-\pi,\pi]$, and thus
\[
	\dist^2(X,\SO(2)) = \min_{R\in\SOn}\norm{X-R}^2 = 4
	\,,
\]
whereas
\[
	\norm{\sqrt{X^TX}-\id_2}^2 = \norm{\id_2-\id_2}^2 = 0
	\,.
\]
For arbitrary $F\in\R^{2\times2}$, the two terms $\WBiot(F)=\norm{\sqrt{F^TF}-\id_2}^2$ and $\dist^2(F,\SOn)$ therefore need to be carefully distinguished. In the following, we will also use the notation
\[
	\Wdist(F) \colonequals \dist^2(F,\SO(2))
\]
for arbitrary $F\in\R^{2\times2}$.

\subsection[A direct elementary calculation of $\Wdist(F)=\distspace^2({F},\SO(2))$]{A direct elementary calculation of \boldmath$\Wdist(F)=\distspace^2({F},\SO(2))$}

Due to our focus on the planar case, we recall the elementary computation of the Euclidean distance $\distspace^2({F},\SO(2))$. Let $F\in\R^{2\times2}$. Since
\begin{equation}
\label{eq:distElementaryStart}
	\Wdist(F) = \distspace^2\left({F},\SO(2)\right)
	= \inf_{R\in \SO(2)}\norm{F-R}^2
	= \inf_{\alpha\in (-\pi, \pi]} \; \Bignorm{F - \matr{
		\cos \alpha & \sin \alpha \\
		-\sin \alpha & \cos \alpha \\
	}}^2
\end{equation}
and
\begin{align}
	\Bignorm{\matr{
		F_{11}-\cos \alpha & F_{12}-\sin \alpha \\
		F_{21}+\sin \alpha & F_{22}-\cos \alpha \\
	}}^2=(F_{11}-\cos\alpha)^2+(F_{12}-\sin \alpha)^2+(F_{21}+\sin \alpha)^2+(F_{22}-\cos \alpha)^2
	\,,
\end{align}
the derivative with respect to $\alpha$ yields the stationarity condition
\[
	(F_{11}+F_{22})\,\sin \alpha+(F_{21}-F_{12})\,\cos \alpha = 0
	\qquad\iff\qquad
	\iprod{ \matr{\sin\alpha\\\cos\alpha},\, \matr{F_{11}+F_{22}\\F_{21}-F_{12}} } = 0
	\,,
\]
which is satisfied if and only if
\[
	\matr{\sin\alpha\\\cos\alpha}
	= \pm \frac{1}{\sqrt{(F_{21}-F_{12})^2+(F_{11}+F_{22})^2}}\,\matr{-(F_{21}-F_{12})\\F_{11}+F_{22}}
	= \pm \frac{1}{\sqrt{\norm{F}^2+2\.\det F}}\, \matr{-(F_{21}-F_{12})\\F_{11}+F_{22}}
	\,.
\]
The minimum is easily seen to be realized by
\[
	\matr{\sin\alpha\\\cos\alpha} = \frac{1}{\sqrt{\norm{F}^2+2\.\det F}}\, \matr{-(F_{21}-F_{12})\\F_{11}+F_{22}}
\]
which, after application to \eqref{eq:distElementaryStart} and some elementary calculations, yields
\begin{equation}
	\Wdist(F) = \distspace^2\left({F},\SO(2)\right) = \inf_{R\in \SO(2)}\norm{ F-R}^2=\norm{F}^2-2\.\sqrt{\norm{F}^2+2\.\det F}+2
\end{equation}
for any $F\in\R^{2\times2}$. Since
\[
	\norm{F}^2 = \iprod{F,F} = \iprod{F^TF,\id_2} = \tr(C)\,, \qquad \det C=(\det F)^2\,,
\]
with $C=F^TF$, the squared distance to $\SO(2)$ can alternatively be expressed as
\begin{align}
	\Wdist(F) &=
	\begin{cases}
		\mathrlap{\tr(C)-2\.\sqrt{\tr(C)+2\.\sqrt{\det C}}+2}
		\hphantom{\lambda_1^2+\lambda_2^2-2\.\sqrt{\lambda_1^2+\lambda_2^2+2\.\lambda_1\lambda_2}+2}
		&\casesif \det F>0\,,\\
		\tr(C)-2\.\sqrt{\tr(C)-2\.\sqrt{\det C}}+2 &\casesif \det F\leq 0
	\end{cases}
	\\[.49em]
	&=\begin{cases}
		\mathrlap{I_1-2\.\sqrt{I_1+2\.\sqrt{I_2}}+2}
		\hphantom{\lambda_1^2+\lambda_2^2-2\.\sqrt{\lambda_1^2+\lambda_2^2+2\.\lambda_1\lambda_2}+2}
		&\casesif \det F>0\,,\\
		I_1-2\.\sqrt{I_1-2\.\sqrt{I_2}}+2 &\casesif \det F\leq 0
	\end{cases}
	\\[.49em]
	&=\begin{cases}
		\lambda_1^2+\lambda_2^2-2\.\sqrt{\lambda_1^2+\lambda_2^2+2\.\lambda_1\lambda_2}+2 &\casesif \det F>0\,,\\
		\lambda_1^2+\lambda_2^2-2\.\sqrt{\lambda_1^2+\lambda_2^2-2\.\lambda_1\lambda_2}+2 &\casesif \det F\leq 0\,,
	\end{cases}\notag
\end{align}
where $\lambda_1,\lambda_2$ denote the singular values of $F$ and $I_1=\lambda_1^2+\lambda_2^2$, $I_2=\lambda_1^2\.\lambda_2^2$ are the invariants of $C$. Note that in contrast, the Biot strain energy $\WBiot$ can easily be expressed in terms of the singular values $\lambda_1,\lambda_2$ of $F$ as
\[
	\WBiot(F) = \norm{\sqrt{F^TF}-\id_2}^2 = (\lambda_1-1)^2 + (\lambda_2-1)^2
\]
for any $F\in\R^{2\times2}$.

\subsection[Equality of $\Wdist$ and $\WBiot$ on $\GLp(2)$]{Equality of \boldmath$\Wdist$ and $\WBiot$ on \boldmath$\GLp(2)$}
The Euclidean distance to $\SO(2)$ and the Biot energy are closely related by their restriction to the group $\GLp(2)$ of invertible matrices with positive determinant. First, we observe that for any $F\in\R^{2\times2}$, the Cayley-Hamilton formula yields
\[
	\norm{U}^2-[\tr(U)]^2+2\.\det U=0\quad\implies\quad \tr(U)=\sqrt{\norm{U}^2+2\.\det U}
	\,,
\]
since $\tr(U)>0$ for $U=\sqrt{F^TF}$. Thus the Biot energy can be expressed as
\begin{align}
	\WBiot(F) = \norm{U-\id_2}^2 &= \norm{U}^2 - 2\.\tr(U) + \norm{\id_2}^2\nonumber\\
	&= \norm{U}^2-2\.\sqrt{\norm{U}^2+2\.\det U}+2\nonumber\\
	&= \tr(C)-2\.\sqrt{\tr(C)+2\.\sqrt{\det C}}+2
	=I_1-2\.\sqrt{I_1+2\.\sqrt{I_2}}+2\label{eq:biot_invariant_representation}
\end{align}
for all $F\in\R^{2\times2}$. In the important special case $F\in\GLp(2)$, we can use the simple observation that
\[
	\det U=\sqrt{\det C}=\det F
\]
for $\det F>0$, to further simplify the Biot energy:
\begin{align*}
	\WBiot(F) &= \norm{U}^2-2\.\sqrt{\norm{U}^2+2\.\det U}+2\\
	&=
	\norm{F}^2-2\.\sqrt{\norm{F}^2+2\.\det F}+2=\left(\sqrt{\norm{F}^2+2\.\det F}-1\right)^2+1-2\.\det F
	\,.
\end{align*}
%
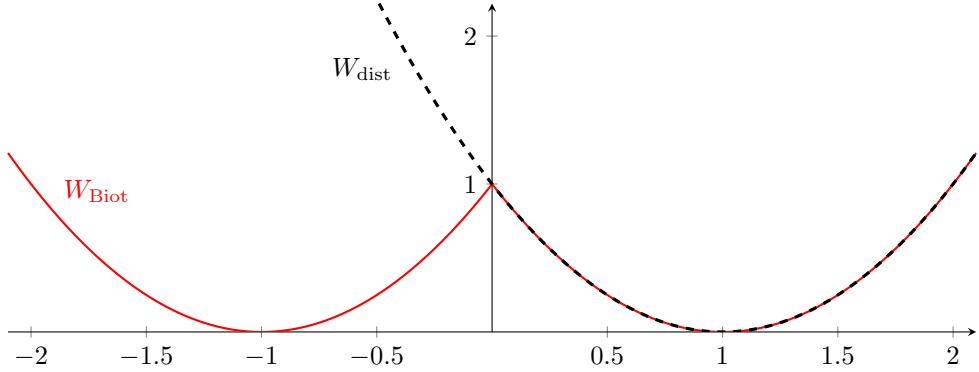
\begin{figure}[h!]
	\centering
		\centering
		\begin{tikzpicture}
			\begin{axis}[
				graphaxis,ytick={1,2},
				width=.833\textwidth,
				height=.343\textwidth]
				\addplot[domain=-2.1:2.1,diststyle]{(abs(x)-1)^2} node[pos=.07,above right] {$\WBiot$};
				\addplot[domain=-.49:2.1,biotstyle]{(x-1)^2} node[pos=.07,below left] {$\Wdist$};
			\end{axis}
		\end{tikzpicture}
	\centering
	\caption{Comparison between the energies $\Wdist(x)=\distspace^2(x,\SO(1))=\distspace^2(x,{1})=\abs{x-1}^2$ (black) and $\WBiot(x)=\bigabs{\abs{x}-1}^2$ (red) in the one-dimensional case.}
	\label{distBiot}
\end{figure}%

\noindent In particular, we note that
\begin{align}\label{eq:equality_on_glp}
	\dist^2(F,\SO(2)) = \Wdist(F) &= \WBiot(F) \qquad\text{for all }\;F\in\GLp(2)\,,
\intertext{while in general,}
	\Wdist(F) &\geq \WBiot(F) \qquad\text{for all }\;F\in\R^{2\times2}
	\,,
\end{align}
as visualized in Fig.~\ref{distBiot}.

\section{The quasiconvex relaxation of Biot-type strain measures}
\label{section:analyticalResults}
As indicated above, when computing the quasiconvex envelopes of the energies $\WBiot$ and $\Wdist$, it is important to distinguish between the unconstrained case, i.e.\ the relaxation of the mappings
\[
	\WBiot,\Wdist\col\;\R^{2\times2}\to\R
\]
over the entire space $\R^{2\times2}$, and the determinant-constrained relaxation with respect to the domain $\GLp(2)$. Furthermore, since the terms $\norm{\sqrt{F^TF}-\id_2}^2$ and $\Wdist(F)=\dist^2(F,\SOn)$ are not equal in general for $F\notin\GLp(2)$, their respective quasiconvex envelopes need to be computed separately, even at $F\in\GLp(2)$, where the energies coincide.

\subsection[The quasiconvex envelope of $\distspace^2({F},\SO(2))$ on $\R^{2\times 2}$]{The quasiconvex envelope of \boldmath$\distspace^2({F},\SO(2))$ on \boldmath$\R^{2\times 2}$}
\label{section:envelopeDistSO2}
First, we will explicitly compute the relaxation of the mapping $F\mapsto\distspace^2({F},\SO(2))$ on $\R^{2\times 2}$. To this end, we recall the Brighi-Theorem, adapted to $n=2$.
\begin{theorem}[{\cite[Theorem 3.2., p. 310]{Brighi97}}]
\label{theorem:brighi}
	Let $q\col\R^{2\times 2} \to\R_+$ be a non-negative quadratic form in $F$ and let $W\col\R^{2\times 2}\to \R$ be given by
	\[
		W(F)=\varphi(q(F)).
	\]
	Then for any $F\in \R^{2\times 2}$,
	\begin{equation}\label{eq:brighi_envelopes}
		q(F)\leq \alpha \qquad \implies \qquad RW(F) = QW(F) = PW(F) = CW(F) = \mu^*
		\,,
	\end{equation}
	where
	\[
		\mu^*=\inf _{t\in \R_+}\varphi(t)=\varphi(\alpha).
	\]
\end{theorem}
\noindent Using Theorem~\ref{theorem:brighi}, the relaxation of $\Wdist$ can be determined directly.
\begin{proposition}
	Let $\Wdist\col\R^{2\times2}\to\R,\,F\mapsto\Wdist(F)=\dist^2(F,\SO(2))$. Then the quasiconvex relaxation of $\dist^2$ over $\R^{2\times2}$ is given by
	\begin{equation}
		Q\Wdist(F) =
		\begin{cases}
			1-2\.\det F &\casesif \norm{F}^2+2\.\det F< 1\,,\\
			(\sqrt{\norm{F}^2+2\.\det F}-1)^2+1-2\.\det F &\casesif \norm{F}^2+2\.\det F\geq 1\,.
		\end{cases}
	 \end{equation}
\end{proposition}
\begin{proof}
	Theorem \ref{theorem:brighi} is directly applicable to $q\colon\R^{2\times 2} \rightarrow\R_+$ with
	\[
		q(F)=\norm{F}^2+2\.\det F
		\,;
	\]
	note that $q$ is indeed a non-negative quadratic form. Consider the function $\varphi\col\R_+\rightarrow\R_+$ with $\varphi(t)=(\sqrt{t}-1)^2$ and let $W(F)=\varphi(q(F))$. Since
	\[
		\inf_{t\in \R_+}\varphi(t) = 0 = \varphi(1) \quad \implies\quad \mu^*=0\,,\;\;\alpha=1\,,
	\]
	the Brighi Theorem yields the implication
	\begin{equation}
	\label{eq:pointwiseEnvelope}
		q(F)\leq 1 \qquad \implies \qquad RW(F) = QW(F) = PW(F) = CW(F) = 0
		\,.
	\end{equation}
	Since $t\mapsto (\sqrt{t}-1)^2$ is convex on $\Rp$ as well as increasing for $t>1$ and $q$ is convex as a non-negative quadratic form, the composition $W(F)=\varphi(q(F))$ is convex for $F$ with $q(F)\geq 1$.  Therefore, the function $\Wtilde\col\R^{2\times2}\to\R$ with
	\[
		\Wtilde(F)
		= \begin{cases}
			\quad0 &\casesif q(F)\leq 1\,,\\
			(\sqrt{q(F)}-1)^2 &\casesif q(F)\geq 1
		\end{cases}
		=\ \begin{cases}
			\quad0 &\casesif \norm{F}^2+2\.\det F< 1\,,\\
			(\sqrt{\norm{F}^2+2\.\det F}-1)^2 &\casesif \norm{F}^2+2\.\det F\geq 1
		\end{cases}
	\]
	is continuous, everywhere convex, below $W$ for $q(F)<1$ where it coincides with the value of the quasiconvex envelope established in \eqref{eq:pointwiseEnvelope}, and equal to $W$ for $q(F)\geq1$, which implies that $\Wtilde=QW$ is the quasiconvex hull of $W$. Finally, using the representation
	\begin{align}
		\Wdist(F) = \distspace^2(F,\SO(2))=\left(\sqrt{\norm{F}^2+2\.\det F}-1\right)^2 + 1 - 2\.\det F = W(F) + 1 - 2\.\det F\,,
	\end{align}
	of $\Wdist$, it is easy to see that
	\begin{align*}
		Q\Wdist(F) &= QW(F) + 1 - 2\.\det F
		\\
		&=\begin{cases}
		1-2\.\det F &\casesif \norm{F}^2+2\.\det F< 1\,,\\
		(\sqrt{\norm{F}^2+2\.\det F}-1)^2+1-2\.\det F &\casesif \norm{F}^2+2\.\det F\geq 1\,,
		\end{cases}
	\end{align*}
	since $F\mapsto 1-2\.\det F$ is a Null-Lagrangian.
\end{proof}
Using a different proof, this result was initially obtained by {\v{S}}ilhav{\`y} \cite{silhavy2001rank}, establishing also that $Q\Wdist$ is indeed a polyconvex function; another proof was given by Dolzmann~\cite[Section 4]{dolzmann2012regularity} (cf.~\cite{agn_martin2015rank}).

\subsubsection{Behaviour for basic deformation modes}
In order to visualize the behaviour of $\Wdist$ and its relaxation, we consider a number of basic deformations and directly compare the corresponding energy function values. First, we note that for purely volumetric deformations of the form $F=\alpha\cdot\id_2$ with $\alpha\in\R$,
\begin{align}
	Q\Wdist(\alpha\cdot \id_2)
	&=\begin{cases}
		1-2\.\det \alpha\cdot \id_2 &\casesif \norm{\alpha\cdot \id_2}^2+2\.\det \alpha\cdot \id_2\leq 1\,, \\
		(\sqrt{\norm{\alpha\cdot \id_2}^2+2\.\det \alpha\cdot \id_2}-1)^2+1-2\.\det \alpha\cdot \id_2 &\casesif \norm{\alpha\cdot \id_2}^2+2\.\det \alpha\cdot \id_2\geq 1
	\end{cases}
	\nonumber\\[.49em]
	&=\begin{cases}
		1-2\.\alpha^2 &\casesif 2\alpha^2+2\.\alpha^2\leq 1\,, \\
		(\sqrt{2\alpha^2+2\.\alpha^2}-1)^2+1-2\.\alpha^2 &\casesif 2\alpha^2+2\.\alpha^2\geq 1
	\end{cases}
	\nonumber
	\\[.49em]
	&=\begin{cases}
		1-2\.\alpha^2 &\casesif \alpha\in \left[-\frac{1}{2},\frac{1}{2} \right]\,, \\
		(2\abs{\alpha}-1)^2+1-2\.\alpha^2 &\casesif \alpha\in\R\setminus \left[-\frac{1}{2},\frac{1}{2} \right]
	\end{cases}
	\nonumber
	\\[.49em]
	&=\begin{cases}
		1-2\.\alpha^2 &\casesif \alpha\in \left[-\frac{1}{2},\frac{1}{2} \right]\,, \\
		2\alpha^2-4\abs{\alpha}+2 &\casesif \alpha\in\R\setminus \left[-\frac{1}{2},\frac{1}{2} \right]\,.
	\end{cases}
\end{align}

\begin{figure}[h!]
	\centering
	\begin{minipage}{0.49\linewidth}
		\centering%
		\begin{tikzpicture}
			\begin{axis}[
				graphaxis,ytick={2,3},
				width=.98\textwidth,
				height=.637\textwidth]
				\addplot[domain=-2.17:-.5,Qdiststyle]{2*x^2-4*abs(x)+2};
				\addplot[domain=-.5:.5,Qdiststyle]{1-2*x^2};
				\addplot[domain=.5:2.17,Qdiststyle]{2*x^2-4*abs(x)+2};
				\addplot[domain=-2.17:2.17,diststyle]{2*(abs(x)-1)^2};
			\end{axis}
		\end{tikzpicture}%
	\end{minipage}%
	\hfill%
	\begin{minipage}{0.49\linewidth}
		\centering
		\begin{tikzpicture}
			\begin{axis}[
				graphaxis,
				width=.98\textwidth,
				height=.637\textwidth]
				\addplot[domain=-2.17:2.17,Qdiststyle]{(sqrt(4+x^2)-1)^2-1};
				\addplot[domain=-2.17:2.17,diststyle]{(sqrt(4+x^2)-1)^2-1};
			\end{axis}
		\end{tikzpicture}
	\end{minipage}%
	\caption{Left: Energy behaviour of $\Wdist$ and its relaxation for volumetric deformations, with the mapping $\alpha\mapsto \Wdist(\alpha\cdot \id_2) = \dist^2(\alpha\cdot\id_2,\SO(2))$ shown in red vs.\ the relaxed energy mapping $\alpha\mapsto Q\Wdist(\alpha\cdot \id_2)$ in blue. Right: For simple shear, the mapping $\gamma\mapsto \Wdist(\matr{1&\gamma\\0&1})$ is equal to the relaxed mapping $\gamma\mapsto Q\Wdist(\matr{1&\gamma\\0&1})$.}
	\label{Figss}
	\label{Qdist}
\end{figure}
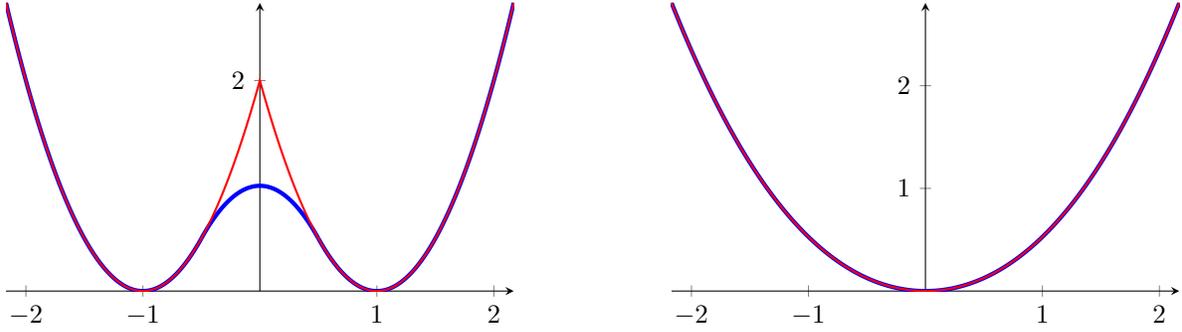%

\noindent In particular,
\begin{equation}\label{eq:compression_energy_values}
	Q\Wdist(0.4\cdot \id_2) = 0.68\,, \qquad \text{whereas}\qquad \Wdist(0.4\cdot \id_2) = 0.72
	\,.
\end{equation}
The difference between the energy $\Wdist$ and its quasiconvex relaxation $Q\Wdist$ for volumetric deformations is shown in Fig.~\ref{Qdist}.

\medskip

\noindent For a simple shear deformation, on the other hand, we find
\[
	\Bignorm{\matr{1&\gamma\\0&1}}^2+2\.\det \matr{1&\gamma\\0&1}
	= 1+1+\gamma^2+2\cdot 1
	= 4+\gamma^2\geq 1
\]
and thus (cf.~Fig.~\ref{Figss})
\begin{align}\label{eq:simple_shear_relaxation_Rnn}
	Q\Wdist(\matr{1&\gamma\\0&1}) &= \Wdist(\matr{1&\gamma\\0&1})\\
	&= (\sqrt{\Bignorm{\matr{1&\gamma\\0&1}}^2+2\.\det \matr{1&\gamma\\0&1}} - 1)^2 + 1 - 2\.\det \matr{1&\gamma\\0&1} = (\sqrt{4+\gamma^2}-1)^2 - 1 \nonumber
\end{align}
for all $\gamma\in\R$.


\medskip

\noindent Finally, for the general case of diagonal matrices in $\R^{2\times2}$, we compute (cf.~Fig.~\ref{fig:energy_surfaces})
\begin{align*}
	Q\Wdist(\matr{\alpha&0\\0&\beta})
	&= \begin{cases}
		1-2\.\alpha\beta &\casesif \alpha^2+\beta^2+2\.\alpha\beta\leq 1\,,\\
		(\sqrt{\alpha^2+\beta^2+2\.\alpha\beta}-1)^2+1-2\.\alpha\beta &\casesif \alpha^2+\beta^2+2\.\alpha\beta\geq 1
	\end{cases}
	\\[.49em]
	&=\begin{cases}
		1-2\.\alpha\beta &\casesif \abs{\alpha+\beta}\leq 1\,,\\
		(\abs{\alpha+\beta}-1)^2+1-2\.\alpha\beta &\casesif \abs{\alpha+\beta}\geq 1\,.
	\end{cases}
\end{align*}
In particular, for $\beta=0$, we find (cf.~Fig.~\ref{b0})
\begin{align*}
	Q\Wdist(\diag(\alpha,0))
	&= \begin{cases}
		1 &\casesif \abs{\alpha}\leq 1\,,\\
		(\abs{\alpha}-1)^2+1 &\casesif \abs{\alpha}\geq 1\,,
	\end{cases}
\intertext{while for $\beta=1$,}
	Q\Wdist(\diag(\alpha,1))
	&= \begin{cases}
		1-2\alpha &\casesif \abs{\alpha+1}\leq 1\,, \\
		(\abs{\alpha+1}-1)^2+1-2\alpha &\casesif \abs{\alpha+1}\geq 1\,.
	\end{cases}
\end{align*}

\newcommand{\sepgraph}[3]{%
	\pgfmathsetmacro{\beta}{#1}%
	\pgfmathsetmacro{\a}{-1-\beta}%
	\pgfmathsetmacro{\b}{1-\beta}%
	\addplot[domain=#2:\a,Qdiststyle]{(abs(x+\beta)-1)^2+1-2*\beta*x};
	\addplot[domain=\a:\b,Qdiststyle]{1-2*\beta*x};
	\addplot[domain=\b:#3,Qdiststyle]{(abs(x+\beta)-1)^2+1-2*\beta*x};
	\addplot[domain=#2:#3,diststyle]{(abs(x+\beta)-1)^2+1-2*\beta*x};
}
\begin{figure}[h!]
	\centering
	\begin{minipage}{0.49\textwidth}%
		\centering
		\begin{tikzpicture}
			\begin{axis}[
				graphaxis,
				ytick={2}, ymin=-.245,ymax=2.401,
				width=.98\textwidth,
				height=.637\textwidth]
				\sepgraph{0}{-2}{2}
			\end{axis}
		\end{tikzpicture}
		\subcaption{$\beta=0$}
	\end{minipage}%
	\hfill%
	\begin{minipage}{0.49\textwidth}%
		\centering
		\begin{tikzpicture}
			\begin{axis}[
				graphaxis, ymin=-.245,
				width=.98\textwidth,
				height=.637\textwidth]
				\sepgraph{1}{-3}{3}
			\end{axis}
		\end{tikzpicture}
		\subcaption{$\beta=1$}
	\end{minipage}%
	\\[1.47em]%
	\begin{minipage}{0.49\textwidth}%
		\centering
		\begin{tikzpicture}
			\begin{axis}[
				graphaxis, ymin=-.245,
				width=.98\textwidth,
				height=.637\textwidth]
				\sepgraph{.25}{-3}{3}
			\end{axis}
		\end{tikzpicture}
		\subcaption{$\beta=\frac14$}
	\end{minipage}%
	\hfill%
	\begin{minipage}{0.49\textwidth}%
		\centering
		\begin{tikzpicture}
			\begin{axis}[
				graphaxis, ymin=-.245,
				width=.98\textwidth,
				height=.637\textwidth]
				\sepgraph{-.5}{-3}{3}
			\end{axis}
		\end{tikzpicture}
		\subcaption{$\beta=-\frac12$}
	\end{minipage}%
	\caption{The mappings $\alpha\mapsto \Wdist(\diag(\alpha,\beta))$ and $\alpha\mapsto Q\Wdist(\diag(\alpha,\beta))$ for fixed $\beta\in\R$ are shown in red and blue, respectively.}
	\label{b0}
\end{figure}
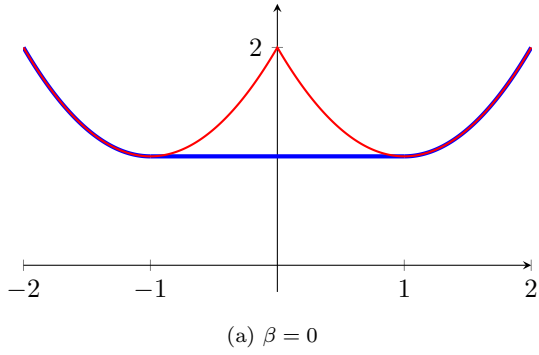
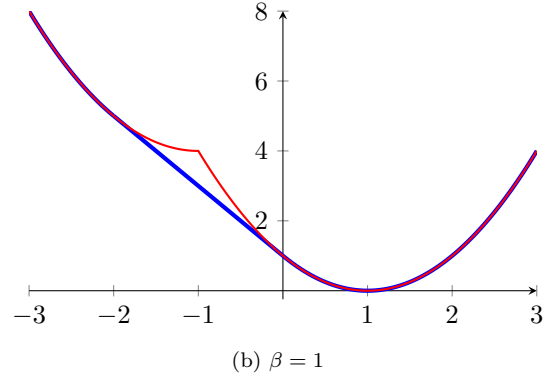
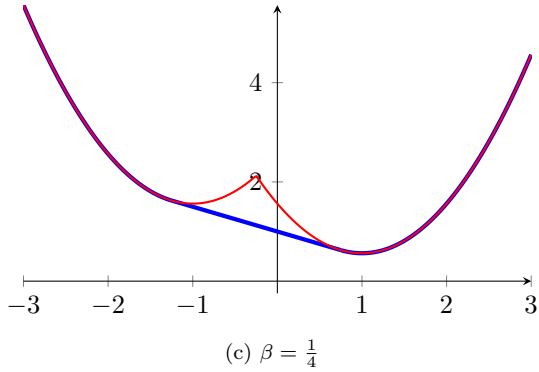
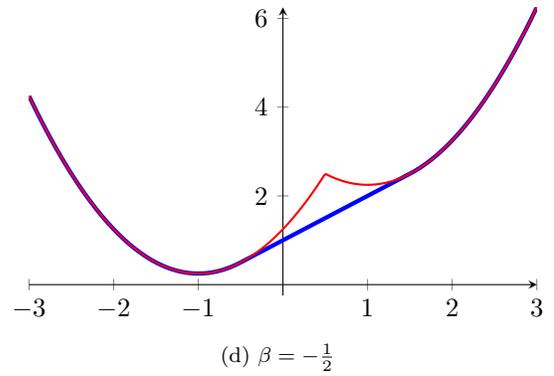%

%
%
%

\clearpage

\subsection[The quasiconvex envelope of $\WBiot$ on $\R^{2\times 2}$]{The quasiconvex envelope of \boldmath$\WBiot$ on \boldmath$\R^{2\times 2}$}

In order to compute the quasiconvex relaxation of the energy function $\WBiot\col\R^{2\times 2}\to\R$ given by $\WBiot(F)=\norm{\sqrt{F^TF}-\id_2}^2$, we will first consider a more general result on energy functions of \emph{Valanis-Landel form} \cite{valanis1967}.\footnote{The so-called \emph{primary matrix functions} on the set $\Symn$ of symmetric $n\times n$-matrices are exactly the gradients of Valanis-Landel-type energies \cite{agn_martin2015some}.} Throughout this section, we will also consider the general case of arbitrary dimension $n\in\N$.

\begin{proposition}%
\label{lemma:valanislandel}
	Let $W\col\Rnn\to\R$ be an objective, isotropic function such that $W$ is of \emph{Valanis-Landel form}, i.e.
	\begin{equation}\label{eq:valanislandel}
		W(F) = \sum_{i=k}^n h(\lambda_k)
	\end{equation}
	for all $F\in\Rnn$ with singular values $\lambda_1,\dotsc,\lambda_n\geq0$, where $h\col[0,\infty)\to\R$ is a continuous function bounded below%
	.
	Then
	\begin{equation}\label{eq:valanis_landel_envelopes}
		RW(F) = QW(F)= PW(F) = CW(F)
		= \sum_{k=1}^n C\htilde(\lambda_k)
	\end{equation}
	for all $F\in\Rnn$ with singular values $\lambda_1,\dotsc,\lambda_n$, where $C\htilde$ denotes the \emph{convex envelope} of
	\begin{equation}
		\label{eq:evencontinuation}
		\htilde\col\R\to\R\,,\quad
		\htilde(t) = \begin{cases}
			h(t) &: t\geq0\,,\\
			h(-t) &: t<0\,.
		\end{cases}
	\end{equation}
\end{proposition}
\begin{proof}
	The rank-one convex envelope $RW$ of $W$ is given by \cite[Theorem 6.10]{Dacorogna08}
	\[
		RW(F)
		= \lim_{k\to\infty}W_k(F)
		\,,
	\]
	where $W_0=W$ and
	\[
		W_{k+1}(F) = \inf \Big\{ t\.W_k(F_1)+(1-t)\.W_k(F_2) \,\setvert\, t\in[0,1]\,,\; t\.F_1+(1-t)\.F_2=F\,,\; \rank(F_2-F_1)=1 \Big\}
	\]
	for all $F\in\Rmn$ and $k\geq1$. By induction, we will first show that for any $W\col\Rnn\to\R$ of the form \eqref{eq:valanislandel} and any $K\leq n$, the estimate
	\begin{equation}\label{eq:valanislandelinduction}
		W_K(\diag(\lambda)) \leq \sum_{k=1}^K C\htilde(\lambda_k) + \sum_{k=K+1}^n \htilde(\lambda_k)
		\qquad\text{for all }\;\lambda=(\lambda_1,\dotsc,\lambda_n)\in\R^n
	\end{equation}
	holds (with possibly empty sums for $K\in\{0,n\}$), where $\htilde$ is given by eq.~\eqref{eq:evencontinuation}.
	
	Inequality \eqref{eq:valanislandelinduction} obviously holds\footnote{Note that $W(\diag(\lambda_1,\dotsc,\lambda_n))=\sum_{k=1}^n h(\abs{\lambda_k})=\sum_{k=1}^n \htilde(\lambda_k)$ for all $\lambda\in\R^n$.} for $K=0$, so assume that \eqref{eq:valanislandelinduction} holds for $K\in\{1,\dotsc,n-1\}$. For $r\in\R$, let $\lambda_i^r$ denote the vector obtained from replacing the $i$-th component of $\lambda$ with $r$, i.e.
	\[
		(\lambda_i^r)_k = \begin{cases}
			r &\casesif k=i\,,\\
			\lambda_k &\casesif k\neq i\,.
		\end{cases}
	\]
	Then for any $a,b\in\R$ and $t\in[0,1]$ with $a\neq b$ and $ta+(1-t)b=\lambda_{K+1}$,
	\begin{itemize}
		\item[i)] $t\.\diag(\lambda_{K+1}^a) + (1-t)\.\diag(\lambda_{K+1}^b) = \diag(\lambda)$,
		\item[ii)] $\rank\bigl(\diag(\lambda_{K+1}^b) - \diag(\lambda_{K+1}^a)\bigr) = 1$
	\end{itemize}
	and therefore
	\begin{align*}
		W_{K+1}(\diag(\lambda))
		&\leq t\.W_K(\diag(\lambda_{K+1}^a))+(1-t)\.W_k(\diag(\lambda_{K+1}^b))
		\\&\leq t\.\Bigl( \sum_{k=1}^K C\htilde((\lambda_{K+1}^a)_k) + \sum_{k=K+1}^n \htilde((\lambda_{K+1}^a)_k) \Bigr) + (1-t)\.\Bigl( \sum_{k=1}^K C\htilde((\lambda_{K+1}^b)_k) + \sum_{k=K+1}^n \htilde((\lambda_{K+1}^b)_k) \Bigr)
		\\&= \sum_{k=1}^K C\htilde(\lambda_k) + t\.\htilde(a) + (1-t)\.\htilde(b) + \sum_{K=K+2}^n \htilde(\lambda_k)
	\end{align*}
	by the induction assumption. In particular,
	\begin{align}
		W_{K+1}(\diag(\lambda))
		&\leq \sum_{k=1}^K C\htilde(\lambda_k) + \sum_{K=K+2}^n \htilde(\lambda_k)
			+ \smash{\underbrace{
				\inf_{\substack{a,b\in\R,\,t\in[0,1]\\ta+(1-t)b=\lambda_{K+1}}} t\.\htilde(a) + (1-t)\.\htilde(b)
			}_{=\.C\htilde(\lambda_{K+1})}}
		\\\nonumber&= \sum_{k=1}^{K+1} C\htilde(\lambda_k) + \sum_{K=K+2}^n \htilde(\lambda_k)
		\,.
	\end{align}
	Due to the assumed isotropy and objectivity of $W$, \eqref{eq:valanislandelinduction} establishes that
	\[
		RW(F) \leq \sum_{k=1}^n C\htilde(\lambda_k) \equalscolon \Wtilde(F)
	\]
	for all $F\in\Rnn$ with singular values $\lambda_1,\dotsc,\lambda_n\geq0$. However, since $C\htilde$ is monotone increasing on $[0,\infty)$, the mapping
	\[
		(\lambda_1,\dotsc,\lambda_n) \mapsto \sum_{k=1}^n C\htilde(\lambda_k)
	\]
	is convex and increasing in each component, which implies that $\Wtilde$ is convex \cite[Theorem 5.1]{ball1976convexity}. Therefore, since $\Wtilde\leq W$,
	\[
		RW(F) \leq \Wtilde(F) \leq CW(F) \leq PW(F) \leq QW(F) \leq RW(F)
	\]
	and thus
	\[
		RW(F) = QW(F)= PW(F) = CW(F) = \Wtilde(F) = \sum_{k=1}^n C\htilde(\lambda_k)
	\]
	for all $F\in\Rnn$ with singular values $\lambda_1,\dotsc,\lambda_n$.
\end{proof}
\begin{corollary}
	A function $W\col\Rnn\to\R$ of the form \eqref{eq:valanislandel} is (rank-one/quasi-/poly-)convex if and only if $h\col[0,\infty)$ is convex and monotone increasing.
\end{corollary}
\begin{proof}
	It is sufficient to observe that the (generalized) convex envelope of $W$ is equal to $W$ if and only if $h=C\htilde$ on $[0,\infty)$, which is the case if and only if $\htilde$ is convex on $\R$ or, equivalently, if $h$ is monotone increasing and convex on $[0,\infty)$.
\end{proof}
Since $\WBiot$ is indeed a function in Valanis-Landel form, i.e.\ a mapping of the form \eqref{eq:valanislandel}, the quasiconvex envelope of $\WBiot$ on $\R^{n\times n}$ can be computed directly from Proposition~\ref{lemma:valanislandel}.
\begin{proposition}
\label{prop:biotRnnEnvelope}
	The quasiconvex envelope of
	\begin{align*}
		&\WBiot\colon\Rnn\to\R\,,\qquad
		\WBiot(F)=\norm{\sqrt{F^TF}-\id_n}^2
	\intertext{on $\Rnn$ is given by}
		Q&\WBiot(F) = \sum_{k=1}^n[\lambda_k-1]^2_+
		\,,\qquad\qquad 
		[x]_+\colonequals
		\begin{cases}
			x &\casesif x\geq 0\,,\\
			0 &\casesif x<0\,,
		\end{cases}
	\end{align*}
	for all $F\in\Rnn$ with singular values $\lambda_1,\dotsc,\lambda_n\geq0$.
\end{proposition}
\begin{proof}
	Since $\WBiot$ is of the form \eqref{eq:valanislandel} with $h(t)=(t-1)^2$, it suffices to observe that for $\htilde(t)=(\abs{t}-1)^2$,
	\[
		C\htilde(t) = \begin{cases}
			0 &\casesif \abs{t}<1\,,\\
			(t-1)^2 &\casesif t\geq1\,,
		\end{cases}
	\]
	and thus $C\htilde(t)=[t-1]_+^2$.
\end{proof}
\noindent An alternative proof of Proposition~\ref{prop:biotRnnEnvelope} is shown in Appendix~\ref{appendix:alternativeProof}.
\begin{figure}[h!]
	\centering
	\begin{tikzpicture}
		\begin{axis}[
			graphaxis,ytick={1,2},
			ymin=-.07,
			width=.735\textwidth,
			height=.343\textwidth]
			\addplot[domain=-2.17:0,plotstyle,color=B]{(-x-1)^2};
			\addplot[domain=0:2.17,plotstyle,color=R]{(x-1)^2};
			%
			\addplot[domain=-2.17:-1,plotstyle,dashed,ultra thick,black]{(-x-1)^2};
			\addplot[domain=-1:1,plotstyle,dashed,ultra thick,black]{0};
			\addplot[domain=1:2.17,plotstyle,dashed,ultra thick,black]{(x-1)^2};
		\end{axis}
	\end{tikzpicture}%
	\caption{\label{fig:envelopes_valanis_landel}The function $h$ on $[0,\infty)$ (red), its even extension $\htilde$ to $\R$ (blue/red) and the convex envelope $C\htilde$ (black dashed line) for the energy $\Wdist$ on $\Rnn$.}%
\end{figure}
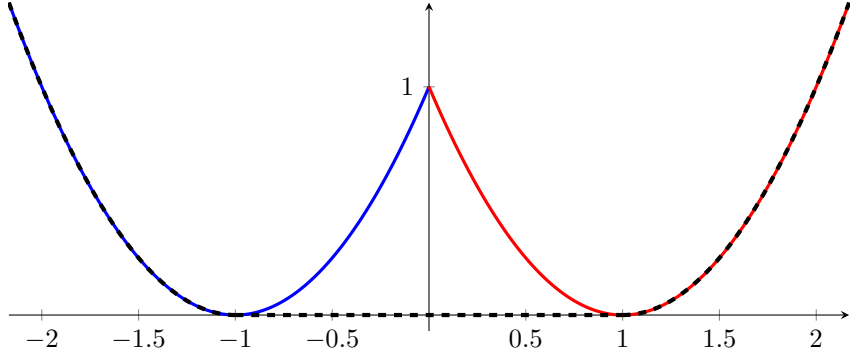%

\subsection[The quasiconvex envelope of the Biot strain measure on $\GLp(2)$]{The quasiconvex envelope of the Biot strain measure on \boldmath$\GLp(2)$}
\label{section:envelopesGlp}

In Section~\ref{section:envelopeDistSO2}, the mapping $F\mapsto \dist^2(F,\SO(2))$ was considered as a function on the entire matrix space $\R^{2\times2}$. In nonlinear elasticity, however, where the assumption $\det F>0$ is required to avoid local self-intersection, it is more natural to consider
the restriction $W\col\GLp(2)\to\R$ of an energy to the group $\GLp(2)$ of matrices with positive determinant or its extension
\begin{equation}\label{eq:inftyExtension}
	\Wtilde\col\R^{2\times2}\to\R\cup\{+\infty\}\,,\qquad
	W(F) =
	\begin{cases}
		W(F) &\casesif F\in\GLp(2)\,,\\
		+\infty &\casesotherwise\,.
	\end{cases}
\end{equation}
In this case, the more appropriate relaxation \cite{pedregal2012parametrized} of the function is given by\footnote{Throughout the following, we will denote by $Q_{\R^{2\times2}}W$ the classical quasiconvex envelope of any function $W\col\R^{2\times2}\to\R$, formerly written as simply $QW$, in order to distinguish it from the constrained envelope $Q_{\GLp(2)}W$.}
\[
	Q_{\GLp(2)}W
	\colonequals
	\sup \{Z\colon\GLp(2)\to \R \setvert Z \;\,\text{quasiconvex}\,,\; Z\leq W\}
	\,.
\]
Since $\WBiot(F)=\Wdist(F)$ for any $F\in\GLp(2)$ according to eq.~\eqref{eq:equality_on_glp}, the quasiconvex envelopes $Q_{\GLp(2)}\WBiot$ and $Q_{\GLp(2)}\Wdist$ must be equal as well. For simplicity, $W$ will therefore denote the restriction of $\WBiot$ or, equivalently, $\Wdist$ to $\GLp(2)$ in the following.
\begin{proposition}
\label{proposition:constrained_envelope}
	Let
	\[
		W\col\GLp(2)\to\R\,,\qquad
		W(F) = \WBiot(F) = \Wdist(F) = \dist^2(F,\SO(2)) = \norm{\sqrt{F^TF}-\id_2}^2
		\,.
	\]
	Then the quasiconvex relaxation $Q_{\GLp(2)}W\col\GLp(n)\to\R$ is given by
	\[
		Q_{\GLp(2)}W(F) = Q_{\R^{2\times2}}\Wdist(F) =
		\begin{cases}
			1-2\.\det F &\casesif \norm{F}^2+2\.\det F< 1\,,\\
			(\sqrt{\norm{F}^2+2\.\det F}-1)^2+1-2\.\det F &\casesif \norm{F}^2+2\.\det F\geq 1\,.
		\end{cases}
	\]
\end{proposition}
\begin{proof}
	Recall that in general, quasiconvexity on $\GLp(2)$ is defined as the quasiconvexity of the extension by $+\infty$ as in eq.~\eqref{eq:inftyExtension}. In particular, the restriction of any quasiconvex function on $\R^{2\times2}$ to $\GLp(2)$ is quasiconvex as well \cite{pedregal2012parametrized}, i.e.
	\[
		\{Z|_{\GLp(2)}\colon\R^{2\times 2}\to \R \setvert Z \;\,\text{quasiconvex}\,,\; Z\leq W\}
		\subset
		\{Z\colon\GLp(2)\to \R \setvert Z \;\,\text{quasiconvex}\,,\; Z\leq \Wdist\}
		\,,
	\]
	which immediately implies that
	\begin{equation}\label{eq:envelopes_constraint_inequality}
		Q_{\GLp(2)}W \geq \sup \{Z\colon\R^{2\times 2}\to \R \setvert Z \;\,\text{quasiconvex}\,,\; Z\leq W\} = Q_{\R^{2\times2}}\Wdist
		\,.
	\end{equation}
	In order to complete the proof,
	it therefore remains to show that $Q_{\GLp(2)}W \leq Q_{\R^{2\times2}}\Wdist$.
	Here, reiterating an excerpt of an earlier proof by Dolzmann \cite{dolzmann2012regularity}, we will show the sufficient condition
	\begin{equation*}
		R_{\GLp(2)}W(F) \leq Q_{\R^{2\times2}}\Wdist(F)
	\end{equation*}
	for $F\in\GLp(2)$, where $R_{\GLp(2)}W $ denotes the rank-one convex envelope of $W$ with respect to $\GLp(2)$.
	
	Without loss of generality,\footnote{Note that $R_{\GLp(2)}W\leq\Wdist(F)=Q_{\R^{2\times2}}\Wdist(F)$ for all $F\in\GLp(2)$ with $\norm{F}^2+2\.\det F\geq1$.} let $F\in\GLp(2)$ with $\norm{F}^2+2\.\det F<1$. We choose
	\[
		H = F\,\matr{0&1\\0&0}
		\,.
	\]
	Then $\rank(H)=1$ and
	\begin{equation}\label{eq:constantDeterminant}
		\det(F+tH) = \det F \cdot \det(\id_2+tH) = \det F
		\,,
	\end{equation}
	thus $F+tH\in\GLp(2)$ for all $t\in\R$. Since $\sqrt{\norm{F}^2+2\.\det F}-1 < 0$ by assumption on $F$, and since
	\[
		\lim_{t\to+\infty} \norm{F+tH}^2 = \lim_{t\to-\infty} \norm{F+tH}^2 = +\infty
		\,,
	\]
	there exist $t_1<0$ and $t_2>0$ such that
	\[
		\sqrt{\norm{F+t_1H}^2+2\.\det F}-1
		= \sqrt{\norm{F+t_2H}^2+2\.\det F}-1
		= 0
	\]
	and thus, due to \eqref{eq:constantDeterminant},
	\begin{align*}
		W(F+t_i H) &= \sqrt{\norm{F+t_i H}^2+2\.\det (F+t_i H)}-1)^2 + 1 - 2\,\det(F+t_i H)
		\\&= \sqrt{\norm{F+t_i H}^2+2\.\det F}-1)^2 + 1 - 2\,\det F
		= 1 - 2\,\det F
	\end{align*}
	for $i\in\{1,2\}$.
	Then with $\lambda=\frac{t_1}{t_2-t_1}\in(0,1)$, we find
	\[
		\lambda(F+t_2H) + (1-\lambda)(F+t_1H) = F + (t_1 - \lambda\.(t_2-t_1)) H = F
	\]
	and therefore, since $F$ is the rank-one convex combination of $F+t_2H\in\GLp(2)$ and $F+t_1H\in\GLp(2)$,
	\begin{equation}\label{eq:rank_one_envelope_inequality}
		R_{\GLp(2)}W(F) \leq \lambda W(F+t_2H) + (1-\lambda) W(F+t_1H) = 1-2\,\det F = Q_{\R^{2\times2}}\Wdist(F)
		\,.
	\end{equation}
	Finally, we combine \eqref{eq:envelopes_constraint_inequality} and \eqref{eq:rank_one_envelope_inequality} to find
	\[
		Q_{\GLp(2)}W(F) \leq R_{\GLp(2)}W(F) \leq Q_{\R^{2\times2}}\Wdist(F) \leq Q_{\GLp(2)}W(F)
	\]
	and thus $Q_{\GLp(2)}W(F) = Q_{\R^{2\times2}}\Wdist(F)$ for all $F\in\GLp(2)$.
\end{proof}

According to Proposition~\ref{proposition:constrained_envelope}, at any $F\in\GLp(2)$, the quasiconvex relaxation of $\Wdist$ with respect to $\R^{2\times2}$ is indeed equal to its relaxation over $\GLp(2)$ only. However, this equality relies on the specific extension of $\Wdist$ to $\R^{2\times2}$ by the mapping $F\mapsto\dist^2(F,\SO(2))$, which is in general not identical to $\WBiot(F)=\norm{\sqrt{F^TF}-\id_2}^2$ for $F\notin\GLp(2)$.

In particular, the quasiconvex relaxation $Q_{\GLp(2)}\WBiot(F)$ of $\WBiot$ over $\GLp(2)$ is \emph{not} identical to its relaxation $Q_{\R^{2\times2}}\WBiot(F)$ over all of $\R^{2\times2}$, even at $F\in\GLp(2)$. For example, consider the case of simple shear (cf.~Fig.~\ref{fig:envelopes_shear}): For $\gamma\in\R$, let
\[
	F_\gamma = \matr{1&\gamma\\0&1}\in\GLp(2)
	\,,\quad\text{ with singular values}\quad
	\lambda_{1} = \frac{\sqrt{\gamma^2+4}+ \abs{\gamma}}{2} \geq 1
	\,,\quad
	\lambda_{2} = \frac{\sqrt{\gamma^2+4}- \abs{\gamma}}{2} \leq 1
	\,.
\]
Then according to
\eqref{eq:simple_shear_relaxation_Rnn} and Proposition~\ref{proposition:constrained_envelope},
\[
	Q_{\GLp(2)}\WBiot(F_\gamma)
	= Q_{\R^{2\times2}}\Wdist(F_\gamma)
	= (\sqrt{4+\gamma^2}-1)^2 - 1
	= \WBiot(F_\gamma)
\]
for all $\gamma\in\R$
but due to
Proposition~\ref{prop:biotRnnEnvelope},
\begin{align*}
	Q_{\R^{2\times2}}\WBiot(F_\gamma)
	= [\lambda_1-1]_+^2 + [\lambda_2-1]_+^2
	&= \Bigl[\frac{\sqrt{\gamma^2+4}+ \abs{\gamma}}{2} - 1\Bigr]_+^2 + \Bigl[\frac{\sqrt{\gamma^2+4}- \abs{\gamma}}{2} - 1\Bigr]_+^2
	\\&= \Bigl(\frac{\sqrt{\gamma^2+4}+ \abs{\gamma}}{2} - 1\Bigr)^2
	\;<\; \WBiot(F_\gamma)
	\quad\text{ for all }\;\gamma\neq0
	\,.
\end{align*}

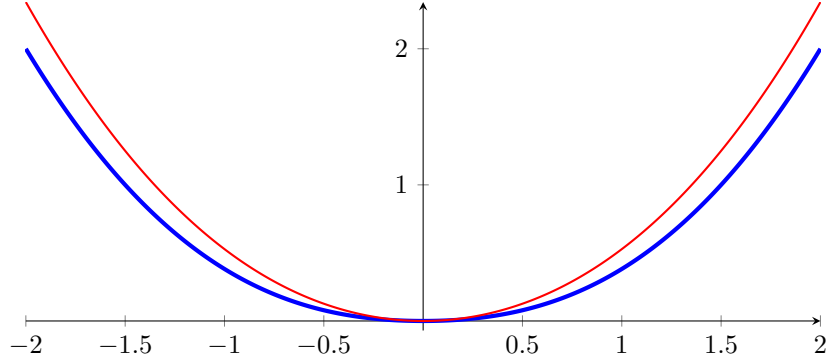
\begin{figure}[h!]
	\centering
	\begin{tikzpicture}
			\begin{axis}[
			graphaxis,ytick={1,2},
			ymin=-.07,
			width=.7\textwidth,
			height=.343\textwidth]
				\addplot[domain=-2:2,Qdiststyle]{((sqrt(x^2+4)+ abs(x))/2 - 1)^2};
				\addplot[domain=-2:2,diststyle]{(sqrt(4+x^2)-1)^2-1};
			\end{axis}
	\end{tikzpicture}%
	\caption{\label{fig:envelopes_shear}The mappings $\gamma\mapsto\WBiot(F_\gamma)=Q_{\GLp(2)}\WBiot(F_\gamma)$ in red and $\gamma\mapsto Q_{\R^{2\times2}}\WBiot(F_\gamma)$ in blue, representing the two different relaxations of $\WBiot$, with and without determinant constraint, in the simple shear case.%
	}%
\end{figure}%

\subsection{Summary}
\label{section:analytical_results_summary}

\noindent The results from Sections \ref{section:envelopeDistSO2}--\ref{section:envelopesGlp} regarding the quasiconvex envelopes of the Biot-type energies $\WBiot$ and $\Wdist$ on $\R^{2\times2}$ and on $\GLp(2)$ can be summarized as follows (cf.~Fig.~\ref{fig:energy_surfaces}): Let
\begin{alignat*}{2}
	&\WBiot\colon &\R^{2\times 2}\to \R\,,\qquad \WBiot(F) &= \norm{\sqrt{F^TF}-\id_2}^2\,,
	\\[.49em]
	&\Wdist\colon &\R^{2\times 2}\to\R\,,\qquad \Wdist(F) &= \dist^2(F,\SO(2))\,.
\end{alignat*}
Then
\begin{alignat}{2}
	&Q_{\R^{2\times2}}\WBiot(F) &&=
		\sum_{k=1}^n[\lambda_k-1]^2_+
		\,,\qquad 
		[x]_+\colonequals
		\begin{cases}
			x &\casesif x\geq 0\,,\\
			0 &\casesif x<0\,,
		\end{cases}
	\\[.49em]
	&Q_{\R^{2\times2}}\Wdist(F) &&=
		\begin{cases}
			1-2\.\det F &\casesif \norm{F}^2+2\.\det F< 1\,, \\
			(\sqrt{\norm{F}^2+2\.\det F}-1)^2+1-2\.\det F &\casesif \norm{F}^2+2\.\det F\geq 1
		\end{cases}
\intertext{for any $F\in\R^{2\times2}$ with singular values $\lambda_1,\lambda_2$, and}
	&\mathrlap{Q_{\GLp(2)}\WBiot(F) =
		Q_{\GLp(2)}\Wdist(F) = Q_{\R^{2\times2}}\Wdist(F)
		} \label{eq:summary_envelope_constrained}
\end{alignat}
for any $F\in\GLp(2)$ with singular values $\lambda_1,\lambda_2$.

\begin{figure}[h!]
	\centering%
	\begin{tikzpicture}
		\begin{axis}[samples=20, view={21}{21}, zmax=10.01,
						colormap={hotCustom}{color=(blue) color=(yellow) color=(yellow) color=(orange) color=(orange) color=(red) color=(red)}
					]
			\addplot3[surf, domain=-2:2, shader=interp] {(abs(x+y)-1)^2 + 1-2*x*y};
			\addplot3[mesh, domain=-2:2, color=black] {(abs(x+y)-1)^2 + 1-2*x*y};
		\end{axis}
		\node[right] at (3.43,4.41){$\Wdist$};
	\end{tikzpicture}%
	\hfill%
	\begin{tikzpicture}
		\begin{axis}[samples=20, view={21}{21}, zmax=10.01,
						colormap={hotCustom}{color=(blue) color=(yellow) color=(yellow) color=(orange) color=(orange) color=(red) color=(red)}
					]
			\addplot3[surf, domain=-2:2, shader=interp] {(x+y)^2<1? 1-2*x*y : (abs(x+y)-1)^2 + 1-2*x*y};
			\addplot3[mesh, domain=-2:2, color=black] {(x+y)^2<1? 1-2*x*y : (abs(x+y)-1)^2 + 1-2*x*y};
		\end{axis}
		\node[right] at (3.43,4.41){$Q_{\R^{2\times2}}\Wdist$};
	\end{tikzpicture}%
	\par%
	\begin{tikzpicture}
		\begin{axis}[samples=20, view={21}{21}, zmax=10.01,
						colormap={hotCustom}{color=(blue) color=(yellow) color=(yellow) color=(orange) color=(orange) color=(red) color=(red)}
					]
			\addplot3[surf, domain=-2:2, shader=interp] {(abs(x)-1)^2+(abs(y)-1)^2};
			\addplot3[mesh, domain=-2:2, color=black] {(abs(x)-1)^2+(abs(y)-1)^2};
		\end{axis}
		\node[right] at (3.43,3.43){$\WBiot$};
	\end{tikzpicture}%
	\hfill%
	\begin{tikzpicture}
		\begin{axis}[samples=20, view={21}{21}, zmax=10.01,
						colormap={hotCustom}{color=(blue) color=(yellow) color=(yellow) color=(orange) color=(orange) color=(red) color=(red)}
					]
			\addplot3[surf, domain=-2:2, shader=interp] {(abs(x)-1<0?0:(abs(x)-1)^2)+(abs(y)-1<0?0:(abs(y)-1)^2)};
			\addplot3[mesh, domain=-2:2, color=black] {(abs(x)-1<0?0:(abs(x)-1)^2)+(abs(y)-1<0?0:(abs(y)-1)^2)};
		\end{axis}
		\node[right] at (3.43,3.43){$Q_{\R^{2\times2}}\WBiot$};
	\end{tikzpicture}
	\par%
	\begin{tikzpicture}[scale=.637]
		\begin{axis}[samples=20, view={21}{21}, zmax=3.43,
						colormap={hotCustom}{color=(blue) color=(yellow) color=(yellow) color=(orange) color=(orange) color=(red) color=(red)}
					]
			\addplot3[surf, domain=0:2, shader=interp] {(abs(x)-1)^2+(abs(y)-1)^2};
			\addplot3[mesh, domain=0:2, color=black] {(abs(x)-1)^2+(abs(y)-1)^2};
		\end{axis}
		\node[right] at (2.94,4.41){\small$\WBiot=\Wdist$};
	\end{tikzpicture}%
	\hfill%
	\begin{tikzpicture}[scale=.637]
		\begin{axis}[samples=20, view={21}{21}, zmax=3.43,
						colormap={hotCustom}{color=(blue) color=(yellow) color=(yellow) color=(orange) color=(orange) color=(red) color=(red)}
					]
			\addplot3[surf, domain=0:2, shader=interp] {(x+y)^2<1? 1-2*x*y : (abs(x+y)-1)^2 + 1-2*x*y};
			\addplot3[mesh, domain=0:2, color=black] {(x+y)^2<1? 1-2*x*y : (abs(x+y)-1)^2 + 1-2*x*y};
		\end{axis}
		\fill[white] (4.9,4.165) rectangle (7,4.606);
		\node[right] at (1.96,4.41){\footnotesize$Q_{\GLp(2)}\WBiot=Q\Wdist$};
	\end{tikzpicture}%
	\hfill%
	\begin{tikzpicture}[scale=.637]
		\begin{axis}[samples=20, view={21}{21}, zmax=3.43,
						colormap={hotCustom}{color=(blue) color=(yellow) color=(yellow) color=(orange) color=(orange) color=(red) color=(red)}
					]
			\addplot3[surf, domain=0:2, shader=interp] {(abs(x)-1<0?0:(abs(x)-1)^2)+(abs(y)-1<0?0:(abs(y)-1)^2)};
			\addplot3[mesh, domain=0:2, color=black] {(abs(x)-1<0?0:(abs(x)-1)^2)+(abs(y)-1<0?0:(abs(y)-1)^2)};
		\end{axis}
		\node[right] at (2.94,4.41){\small$Q_{\R^{2\times2}}\WBiot$};
	\end{tikzpicture}%
	\caption{\label{fig:energy_surfaces}The energy functions $\WBiot$ and $\Wdist$ and their relaxations on the set $\{\diag(x,y)\setvert (x,y)\in\R^2\}$ of diagonal matrices. Note again that $Q_{\GLp(2)}\WBiot=Q_{\GLp(2)}\Wdist=Q_{\R^{2\times2}}\Wdist\neq Q_{\R^{2\times2}}\WBiot$.}
\end{figure}
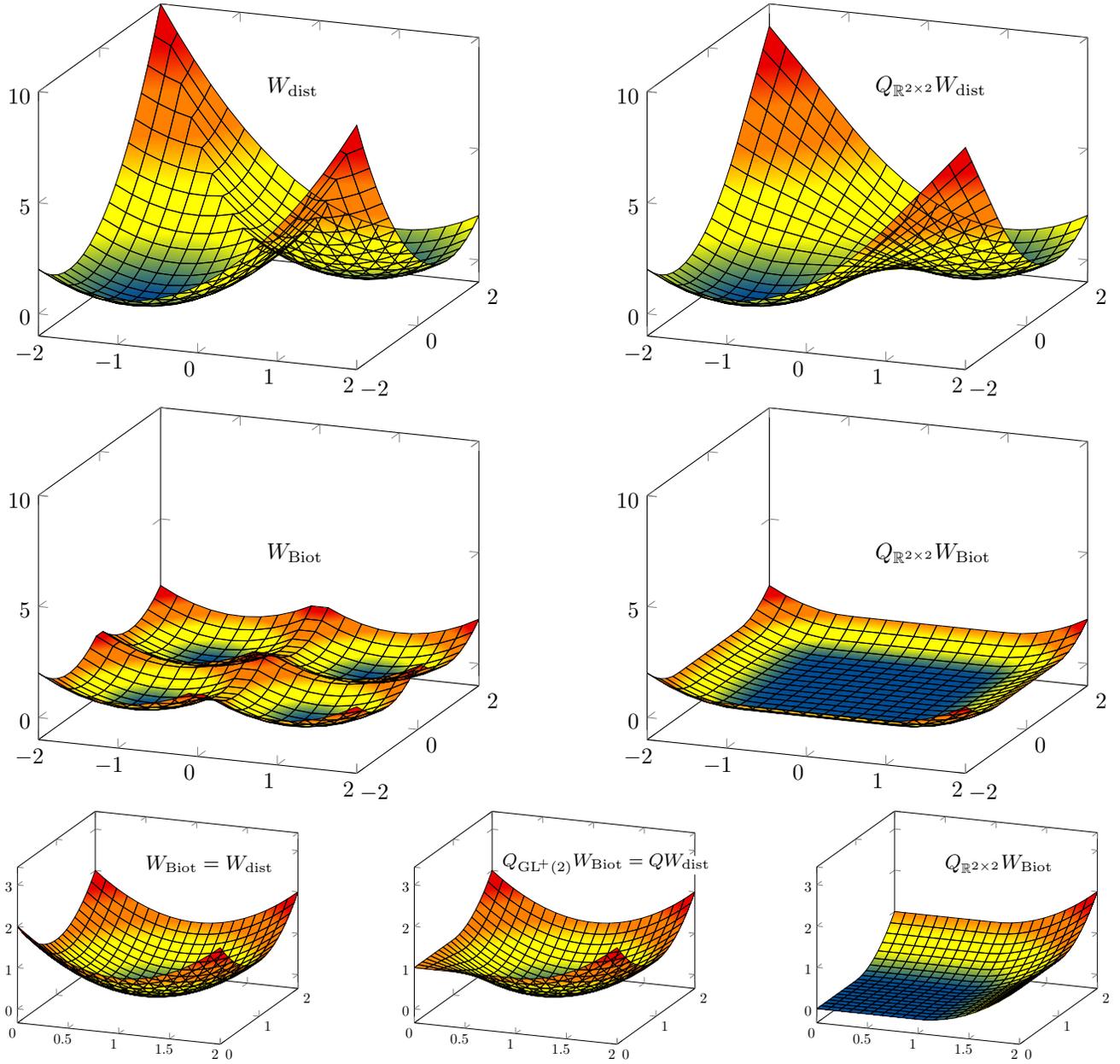

\section{Numerical relaxation results}
\label{sec:numerical_results}
\let\boldsymbolold\boldsymbol%
\renewcommand{\boldsymbol}[1]{#1}%

Due to the explicit analytical results provided in Section~\ref{section:analyticalResults}, the energy functions $\WBiot$ and $\Wdist$ can serve as a benchmark for numerical relaxation methods, since it is possible to compare the numerically obtained energy levels directly to their analytically optimal counterparts. The difference between the envelopes on the domains $\R^{2\times2}$ and $\GLp(2)$, in particular, allows for the separate evaluation of techniques in the determinant-constrained and the unconstrained case.

In the following, we will apply three numerical relaxation methods to the Biot-type energies: two methods for finding the quasiconvex envelope via energy-minimizing deformations under homogeneous boundary conditions and a lamination technique for determining the rank-one convex envelope.

\subsection{Relaxation in a finite element space}
\label{section:FEM}

For any $F\in\R^{2\times2}$ with $Q\WBiot(F) < \WBiot(F)$, which, in particular, is the case for \emph{compressions} with sufficiently large strains \cite{li1996numerical,Kumar04}, the affine linear deformation $x\mapsto Fx$ is no longer energy-optimal under the corresponding homogeneous Dirichlet boundary conditions. Therefore, a material is expected to exhibit microstructures \cite{bartels2006relaxation} under such boundary conditions. In addition to the rank-one convexification considered in Section~\ref{sec:rank-one_convexification}, we will demonstrate the emergence of such microstructures when minimizing the energies in a space of first-order Lagrange finite elements, using a trust-region optimization solver \cite{conn_gould_toint:2000}.

Choosing $\Omega=[-1,1]^2$ and the homogeneous Dirichlet boundary conditions $\phi(x)=F_0\.x$ for $x\in\partial\Omega$ with $F_0=0.4\cdot\id_2$, we numerically minimize the strain energy
\[
	I(\varphi) = \int_\Omega W(\grad\varphi(x))\,\dx
\]
on a grid of $20\times20$ quadrilateral elements. Figure~\ref{fig:fem_results} shows the resulting deformations and the energy functional values for $W=\WBiot$ and $W=\Wdist$ without any determinant constraints, as well as for $W=\WBiot$ with a penalty term for negative determinants:
\[
	\WBiot^+(F)
	= \WBiot(F) + k\.[-\det F]_+
	= \begin{cases}
		\norm{\sqrt{F^T F}-\id_2}^2 &\casesif \det F > 0\,,\\
		\norm{\sqrt{F^T F}-\id_2}^2 - k\.\det F &\casesif \det F < 0\,,
	\end{cases}
\]
where we choose $k=10^5$. In all cases, the energy value at the minimizer is significantly below the energy $\meas{\Omega}\cdot W(F_0)=2.88$ for the homogeneous deformation.

For $\WBiot$ without penalization, the deformation shows clear self-intersection; note that the range is not contained in the homogeneously deformed square $0.4\cdot\Omega$ and that the gradients near the boundary are not orientation-preserving.

On the other hand, no intersections can be observed for the numerical minimizer $\varphi_\subtext{dist}$ of $\Wdist$; in particular, $\det\grad\varphi_{\subtext{dist}}>0$ everywhere. Similarly, the determinant penalization included in $\WBiot^+$ prevents any self-intersection. Consequently, since $\WBiot(F)=\Wdist(F)$ for all $F\in\GLp(2)$, the deformations (and thus the energy levels) for the minimizer of $\Wdist$ and the (penalized) minimizer of $\WBiot^+$ are numerically identical. We also note that for $\Wdist$, the achieved numerical energy value is close to the infimum $\meas{\Omega}\cdot Q\Wdist(F_0)=4\cdot0.68=2.72$, cf.~eq.~\eqref{eq:compression_energy_values}.

\begin{figure*}[h]
    \centering
    \vspace*{1.47em}
    \begin{subfigure}[b]{0.315\textwidth}
        \centering
        \includegraphics[width=\textwidth]{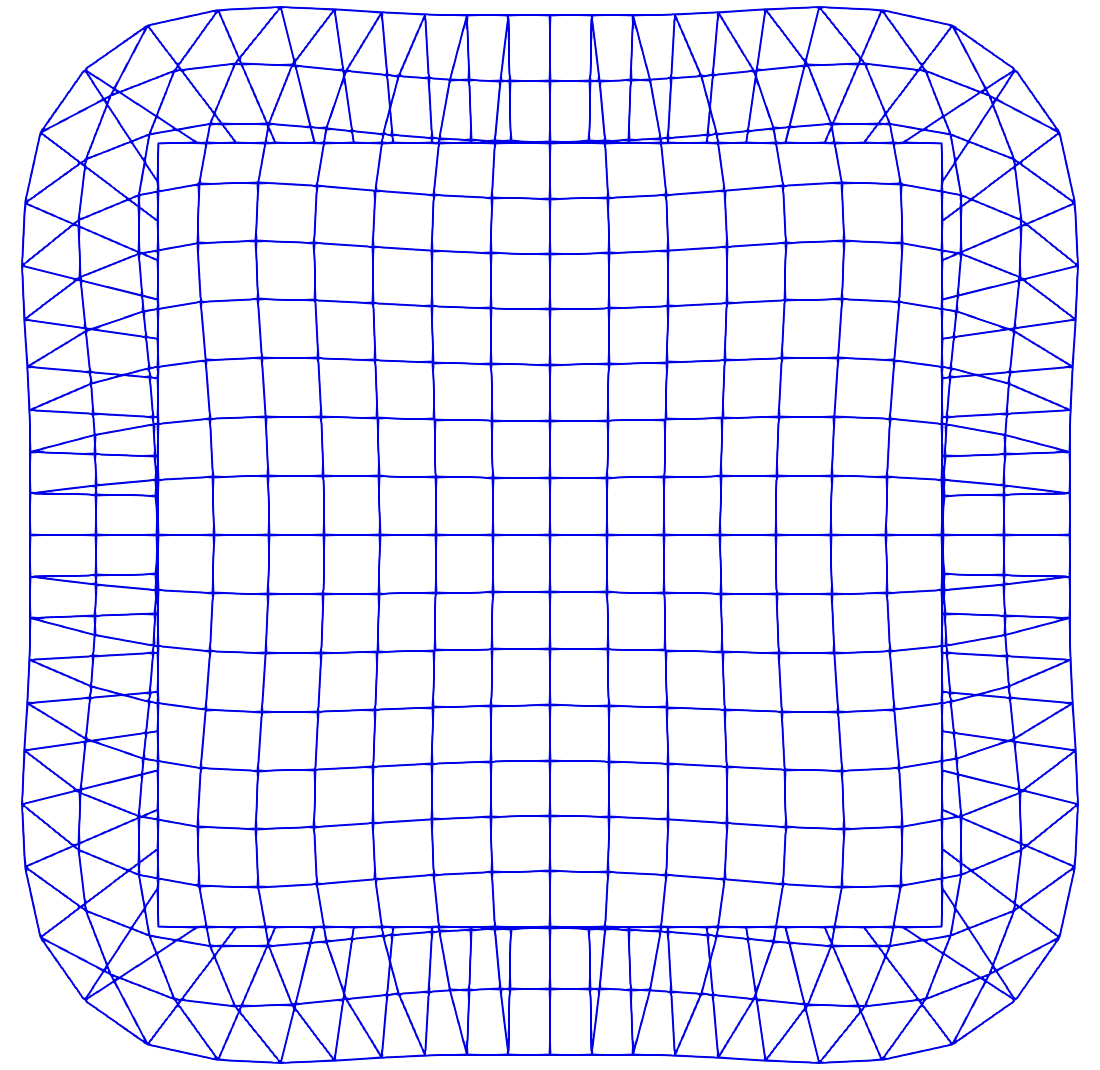}
        \caption{$\WBiot\colon\;I=1.50965$}
        \label{fig:fem-WBiot}
    \end{subfigure}
    \hfill
    \begin{subfigure}[b]{0.315\textwidth}
        \centering
        \includegraphics[width=\textwidth]{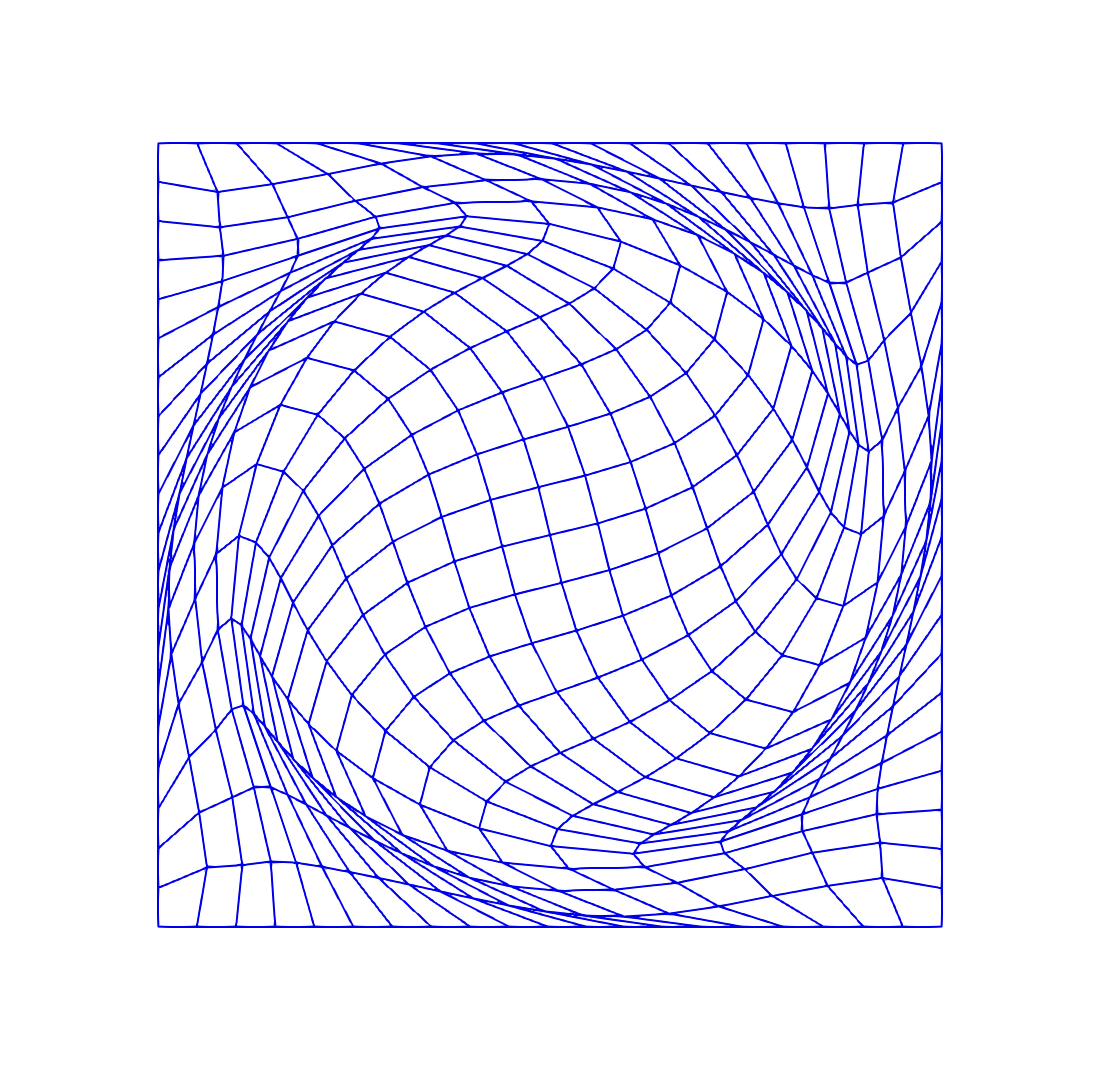}
        \caption{$\Wdist\colon\;I=2.72298$}
        \label{fig:fem-Wdist}
    \end{subfigure}
    \hfill
    \begin{subfigure}[b]{0.315\textwidth}
        \centering
        \includegraphics[width=\textwidth]{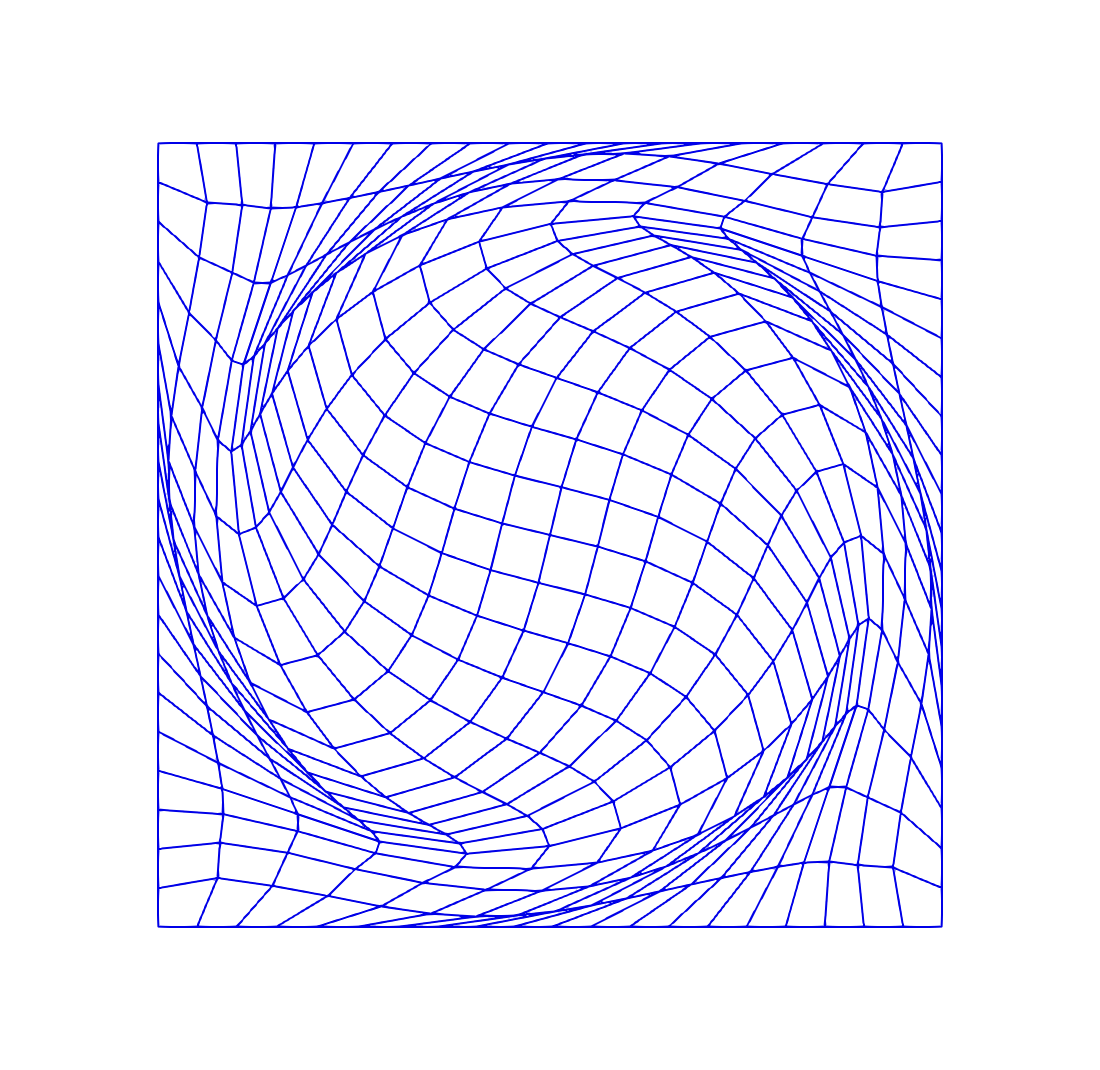}
        \caption{$\WBiot^+$: $I=2.72298$}
        \label{fig:fem-WBiot_constrained}
    \end{subfigure}
    \caption{\label{fig:fem_results} Resulting deformations and energy functional values after trust-region FEM minimization with the energy densities $\WBiot$, $\Wdist$ and $\WBiot^+$, i.e.\ $\WBiot$ with a determinant penalization, under the homogeneous Dirichlet boundary condition $\varphi(x)=0.4\.x$.}
\end{figure*}

%

\clearpage

\subsection{Physics-informed neural networks}
\label{section:PINN}

The energy functionals induced by the energy densities $\WBiot$ and $\Wdist$ can also be minimized numerically using \emph{physics-informed neural networks} (PINNs) \cite{lagaris1998artificial,raissi2019physics,agn_voss2021morrey2}. In the following, we will consider a simple feed-forward neural network of the form
\[
	\Phi\col\R^2\to\R^2\,,\quad\;\;
	\Phi = \Lcal^{3}\circ\phi\circ\Lcal^{2}\circ\phi\circ\Lcal^{1}
	\,,
\]
where $\phi\col\R^n\to\R^n$ with $\phi(x)_i = \tanh(x_i)$
is the $\tanh$ activation function and
$
	\Lcal^i(x) = A^ix+b^i
$
represents the $i$-th fully connected layer; here, $A^1\in\R^{2\times256}$, $A^2\in\R^{256\times256}$, $A^3\in\R^{256\times2}$, $b_1,b_2\in\R^{256}$ and $b_3\in\R^2$ are the free parameters. In order to enforce the Dirichlet boundary condition $\varphi(x)=F_0\.x$ on $\partial\Omega$, the deformation $\varphi$ is given by
\[
	\varphi\col\Omega\to\R^2\,,\quad\;\;
	\varphi(x) = \psi(x)\cdot \Phi(x) + F_0\.x
	\,,
\]
where $\psi\in C^0(\overline{\Omega})$ is chosen such that
\[
	\psi(x) = 0 \quad\iff\quad x\in\partial\Omega
	\,.
\]
The total energy is then approximated via simple Monte-Carlo integration, using $10^5$ points randomly sampled from a uniform distribution over $[-1,1]^2$, and the parameter optimization is performed using the \emph{Adam} gradient descent algorithm \cite{KingmaEtAl2017} with a learning rate of $10^{-3}$. The deformation gradient $F=\grad\varphi$ required to compute the energy $W(F)$ at each sampled point is obtained by automatic differentiation via the \emph{TensorFlow} framework.

We consider three distinct minimization problems on the domain $\Omega=[-1,1]^2$ with the homogeneous boundary condition given by $F_0=0.4\cdot\id_2$: first, we minimize the energy $W=\WBiot$ without additional constraints; next, we consider the modified energy function $W=\WBiot^+\col\R^{2\times2}\to\R$ with
\begin{align}
	\WBiot^+(F) &= \WBiot(F) + k\cdot [-\det(F)]_+^2
	\\
	&=\begin{cases}
		\norm{\sqrt{F^T F}-\id_2}^2 &\casesif \det F>0\,,
		\\
		\norm{\sqrt{F^T F}-\id_2}^2 + k\.(\det F)^2 &\casesif \det F\leq 0
	\end{cases}
\end{align}
with $k=10^5$; finally, we directly minimize the energy $W=\Wdist$. The optimization was performed for up to $500$ epochs (cf.~Fig.~\ref{fig:pinn_training}) and yielded energy levels significantly below the energy $W(F_0)$ of the homogeneous deformation, although the infimum $QW(F_0)=0$ is not approached in the case of the unconstrained relaxation of the energy $\WBiot$.
The resulting deformations are shown in Figs.~\ref{fig:pinn_unconstrained}--\ref{fig:pinn_dist}.

\begin{figure}[h]%
	\begin{center}%
		\begin{minipage}[t]{.49\textwidth}%
			\centering%
			\includegraphics[height=.7\textwidth]{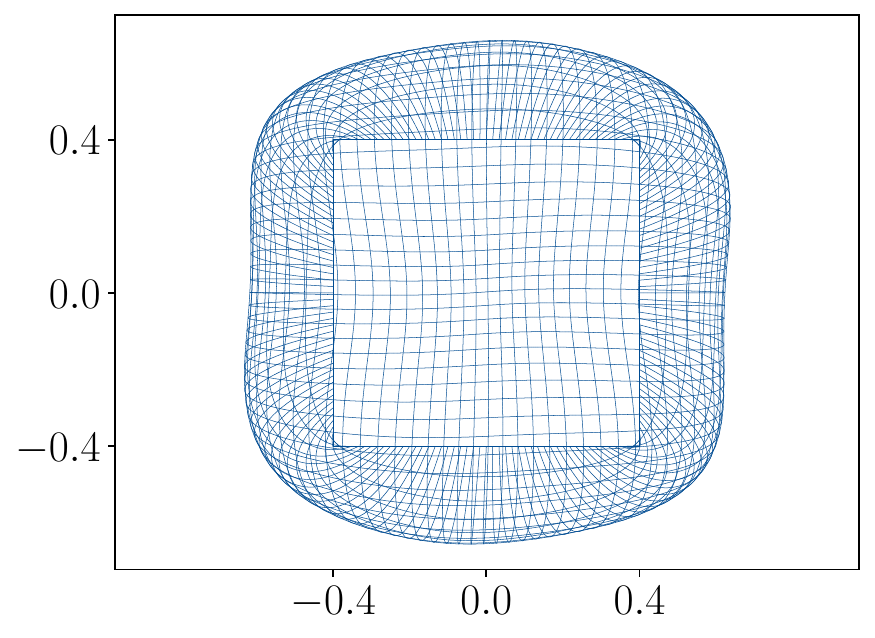}%
			\captionof{figure}{\label{fig:pinn_unconstrained}PINN deformation for the Biot energy $\WBiot$ without determinant constraints or penalties, corresponding to the energy level $I=1.4127$.}%
		\end{minipage}
		\hfill%
		\begin{minipage}[t]{.49\textwidth}%
			\centering%
			\includegraphics[height=.7\textwidth]{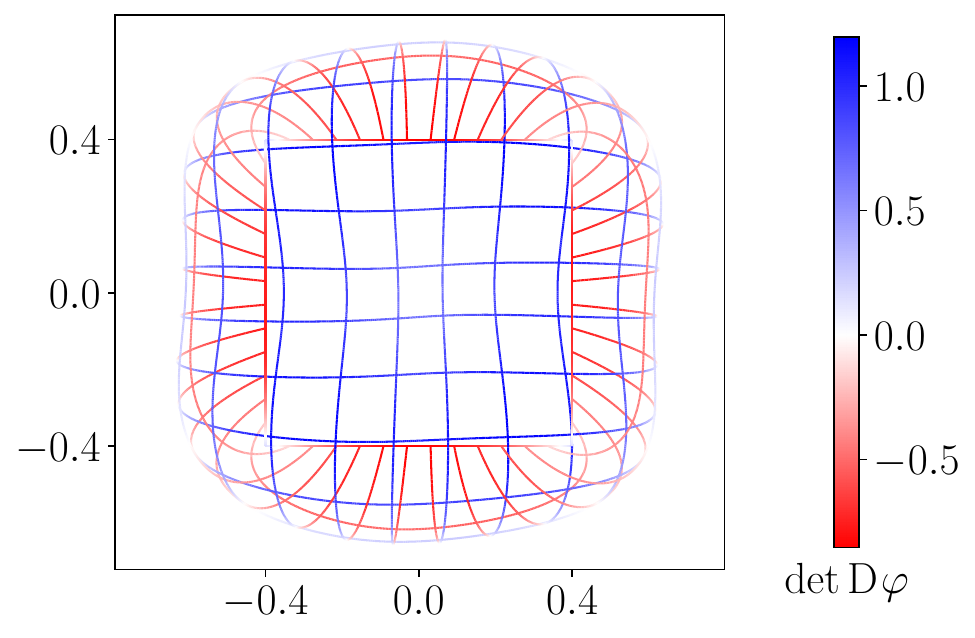}%
			\captionof{figure}{\label{fig:pinn_self_intersection}The unconstrained minimization of $\WBiot$ leads to negative determinants of the deformation gradient and to finite self-intersections.}%
		\end{minipage}%
	\end{center}%
\end{figure}%

\clearpage

\begin{figure}[h]%
	\vspace*{-1.47em}
	\begin{center}%
		\begin{minipage}[t]{0.49\textwidth}%
			\centering%
			\includegraphics[width=\textwidth]{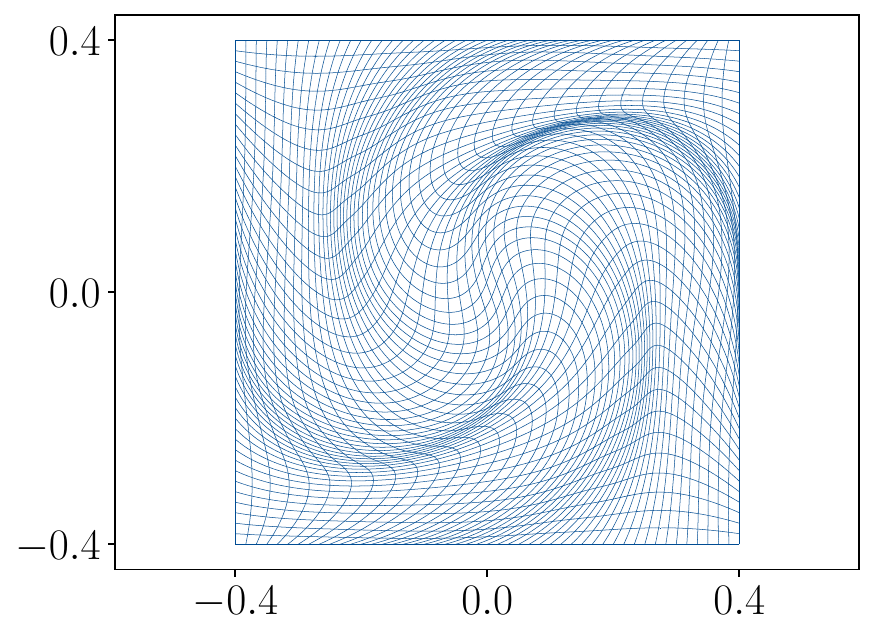}%
			\vspace*{-.49em}%
			\captionof{figure}{\label{fig:pinn_constrained}PINN deformation for the energy $\WBiot^+$, i.e.\ $\WBiot$ with an additive determinant penalty, with the energy level $I=2.7243$ on $[-1,1]^2$.}%
		\end{minipage}%
		\hfill%
		\begin{minipage}[t]{0.49\textwidth}%
			\centering%
			\includegraphics[width=\textwidth]{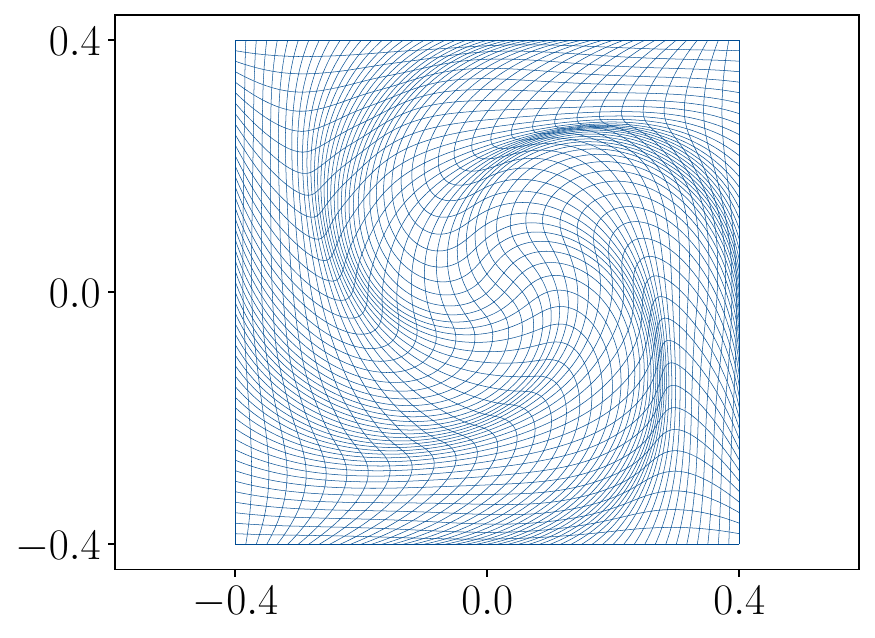}%
			\vspace*{-.49em}%
			\captionof{figure}{\label{fig:pinn_dist}PINN deformation for the energy $\Wdist$ without determinant constraints or penalties with the energy level $I=2.7241$ on $[-1,1]^2$.}%
		\end{minipage}%
		\vspace*{-.49em}%
	\end{center}%
\end{figure}%

The results are qualitatively similar to the FEM results shown in Figure~\ref{fig:fem_results}. In particular, under the deformation induced by the pure Biot energy $\WBiot$, the body self-intersects significantly, as shown in Fig.~\ref{fig:pinn_self_intersection}. However, introducing the penalty term $k\cdot[\det F]_+^2$ with sufficiently large $k>0$ effectively enforces the constraint $\det F > 0$ and thus prevents self-intersection, as shown in Fig.~\ref{fig:pinn_constrained}. Finally, a very similar result to the minimization of $\WBiot^+$ is obtained with the energy $\Wdist$, without any additional penalty terms. As in the FEM computations, self-intersection is not a priori ruled out in this case, but does not occur. Note again that in contrast to $\WBiot$, the globally infimal energy for $\Wdist$ can be reached without local self-intersection; recall that $\Wdist=\WBiot$ on $\GLp(2)$ and thus, according to \eqref{eq:summary_envelope_constrained},
	\begin{alignat}{2}
		\inf_{\varphi\in\adm} \int_\Omega \Wdist(\grad\varphi(x))\,\dx
		&= \meas{\Omega}\cdot Q_{\R^{2\times2}}\Wdist(F_0)
		\nonumber\\
		&= \meas{\Omega}\cdot Q_{\GLp(2)}\WBiot(F_0)
		&&= \inf_{\substack{\varphi\in\adm\\\det \grad\varphi >0\text{ a.e.}}} \int_\Omega \WBiot(\grad\varphi(x))\,\dx\label{eq:self_intersection_no_advantage}\\
		&&&= \inf_{\substack{\varphi\in\adm\\\det \grad\varphi >0\text{ a.e.}}} \int_\Omega \Wdist(\grad\varphi(x))\,\dx\nonumber
		\,,
	\end{alignat}
where $\adm=\{\varphi+F_0 \setvert \in \W1\infty_0(\Omega;\R^2)\}$. Therefore, any self-intersection would not lead to a lower energy level, although \eqref{eq:self_intersection_no_advantage} does not rule out a self-intersecting deformation with the same energy as the one shown in Fig.~\ref{fig:pinn_dist}.

\begin{figure}[h]%
	\begin{center}%
		\includegraphics[width=.441\textwidth]{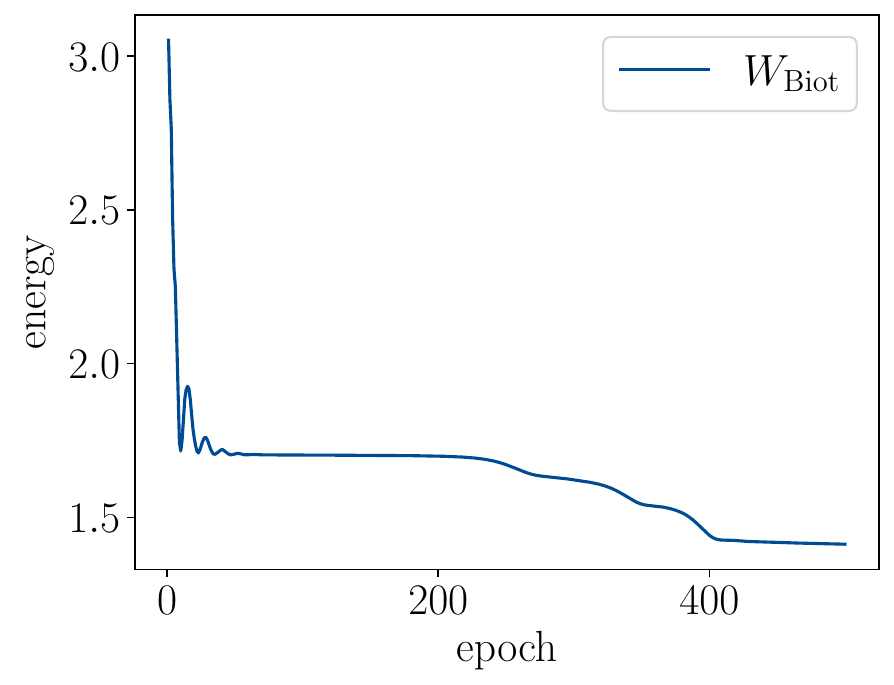}%
		\qquad%
		\includegraphics[width=.441\textwidth]{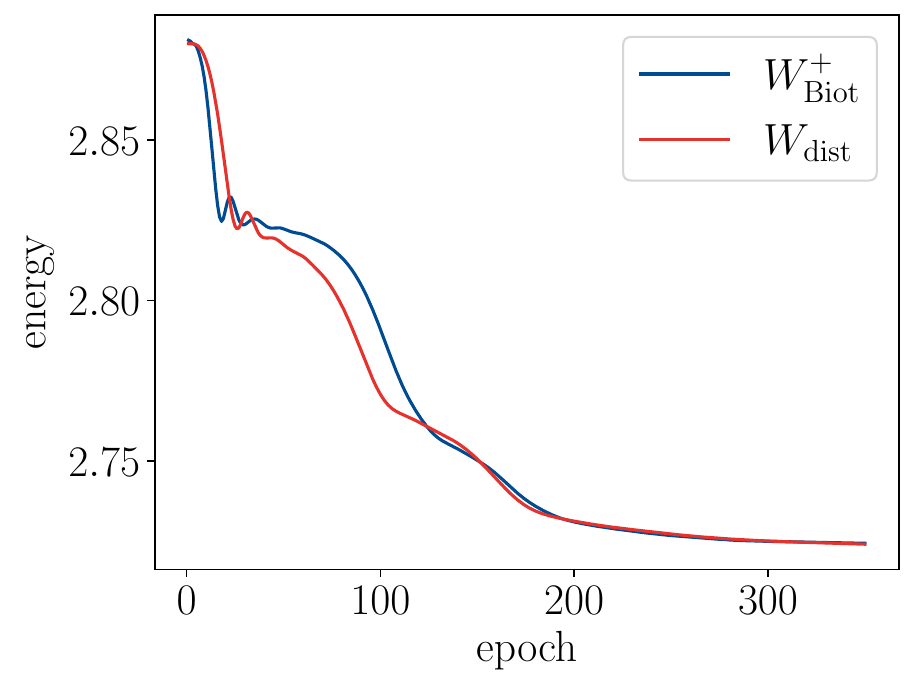}%
		\vspace*{-1.47em}%
	\end{center}%
	\caption{\label{fig:pinn_training}Energy values during $100$ epochs of training the neural network with different energy loss functions.}%
	\vspace*{-3.43em}
\end{figure}%

\subsection{Rank-one convexification}
\label{sec:rank-one_convexification}

Finally, we numerically determine the relaxation of the Biot-type energy via its rank-one convex envelope; cf.~eqs.~\eqref{eq:brighi_envelopes} and \eqref{eq:valanis_landel_envelopes} and recall that in general $QW\leq RW$ for any real-valued function $W$ on $\Rnn$ or $\GLpn$. We briefly recall the algorithm presented in \cite{BalKohNeuPetPet:2022:mrc}, abbreviated as ROC in the following, which is methodologically rooted in \cite{Bar:2004:lca},\cite[Chapter 9.3.4]{Bar:2015:nmn} and \cite{DolWal:2000:ena}.
The energy density $W\colon\R^{2\times2} \to \R$ is discretized as a piecewise linear multivariate interpolation $W_{\delta,r}$ on a grid
\begin{equation} \label{eq:mesh}
	\mathcal{N}_{\delta, r} 
	= \delta \, \mathbb{Z}^{2\times2} \cap \left\{\boldsymbol{A}\in \R^{2\times2} \,\big\vert\, |\boldsymbol{A}|_\infty \leq r\right\}
\end{equation}
that includes arbitrary points in $\R^{2\times 2}$ and thus does not restrict the grid points to be within $\GLp(2)$.
The constraint to $\GLp(2)$ will be discussed later.
The grid has a width parameter \(\delta \in \R\) and a radius parameter \(r \in \R\).
In contrast to the convexification of inelastic incremental stress potentials, e.g.\ damage [4], i.e.\ stress-softening \cite{schmidt2016relaxed} or strain-softening via reconvexification \cite{kohler2023evolving}, the convexification in elasticity is based on strain energy densities which do not change by an evolution of internal variables. Therefore,
it is a reasonable strategy to estimate the non-convex regime, and thus associated numerical parameters, by plotting the original energy density, cf. Fig.~\ref{fig:biot-hseq-R}.
When considering $\WBiot$, we choose the discretization radius $r=2.0$ and the grid width $\delta=0.1$ for the relaxation on $\R^{2\times2}$.
In the case of relaxation on $\GLp(2)$ we restrict the discretization space to the compression regime for the diagonal components of the deformation gradient to $[\delta,1]$ and for the off-diagonal components to $[-1,1]$.

As shown in \cite{BalKohNeuPetPet:2022:mrc} a small rank-one direction set is often sufficient for relaxation.
Thus, we use the set
\begin{equation} \label{eq:bartelsdirections}
	\mathcal{R}^1_{l=1} = \{\delta \, \boldsymbol{a} \otimes \boldsymbol{b} \setvert \boldsymbol{a},\boldsymbol{b} \in \mathbb{Z}^d,\; \abs{\boldsymbol{a}}_\infty,\abs{\boldsymbol{b}}_\infty \leq l\},
\end{equation}
with the shorthand notation \(\mathcal{R}\).
The algorithm is initialized by \(W_{\delta, r}^{0}(\boldsymbol{F}) = W(\boldsymbol{F})\) for all \(\boldsymbol{F} \in \mathcal{N}_{\delta, r}\).
For all points in $\mathcal{N}_{\delta,r}$, the linear optimization problem
\begin{equation} \label{eq:minglobiter}
	W_{\delta, r}^{k+1} (\boldsymbol{F}) = \inf_{\boldsymbol{R},l_1,l_2,\lambda} 
	\left\{\lambda \, W^{k}_{\delta, r} (\boldsymbol{F} + \delta l_1 \boldsymbol{R}) + (1 - \lambda) \, W^{k}_{\delta, r} (\boldsymbol{F} + \delta l_2 \boldsymbol{R}) \, \Big\vert \, 
	\substack{\boldsymbol{R} \in \mathcal{R}, \, \lambda \in [0,1], \, l_1, l_2 \in \mathbb{Z}, \\
	\lambda \, l_1 + (1 - \lambda) \, l_2 = 0}
	\right\}
\end{equation}
is then solved. Points that are not part of $\mathcal{N}_{\delta,r}$ are interpolated.
In eq.~\eqref{eq:minglobiter}, the constraints on $l_1,l_2$ and $\lambda$ ensure that the convex combination leads to $\boldsymbol{F}$.
As the convergence criterion we use a maximum number $k_{\subtext{max}}$ of iterations.

To incorporate the determinant constraint for restricting the relaxation to $\GLp(2)$, we change the evaluation of the line that is to be convexified for a given rank-one direction $\boldsymbol{R} \in \mathcal{R}$ through the point $\boldsymbol{F}$ to
\begin{align} \label{eq:rankOneLineGLp}
	\ell_{\delta,+}(\boldsymbol{F}, \boldsymbol{R}) = \left\{\boldsymbol{F} + l \,\delta \,\boldsymbol{R} \mid l \in \mathbb{Z} \right\} \cap \operatorname{conv}(\mathcal{N}_{\delta, r}) \cap \GLp(2)\,,
\end{align}
where $\operatorname{conv}(\mathcal{N}_{\delta, r})$ denotes the convex hull of the set $\mathcal{N}_{\delta, r}$, so that for fixed iteration $k$, grid point $\boldsymbol{F} \in \mathcal{N}_{\delta,r}$ and direction $\boldsymbol{R} \in \mathcal{R}$, the linear approximation of the energy density along the line $\ell_{\delta,+}$ is evaluated for the relaxation on $\GLp(2)$.
For the relaxation on $\R^{2\times2}$ we use the same line definition as in \cite{BalKohNeuPetPet:2022:mrc}, i.e.
\begin{align} \label{eq:rankOneLine}
	\ell_{\delta}(\boldsymbol{F}, \boldsymbol{R}) = \left\{\boldsymbol{F} + l \,\delta \,\boldsymbol{R} \mid l \in \mathbb{Z} \right\} \cap \operatorname{conv}(\mathcal{N}_{\delta, r}).
\end{align}
This one-dimensional evaluation along the specific rank-one direction~$\boldsymbol{R}$ is then convexified by a computational geometry algorithm, namely, Graham's scan \cite{Gra:1972:ead}, which was first applied in the context of continuum damage mechanics with an adaptive extension in \cite{KohNeuMelPetPetBal:2022:acm}.
Again, the energy density is interpolated whenever points are encountered that are not part of the grid.
After iterating over all directions $\boldsymbol{R}$ the line that yields the lowest relaxed value is selected for the current lamination iteration $k$.

In \cite{KohNeuPetPetBal:2024:hrs} a new algorithm, called HROC, was presented which, in the general case, lacks the convergence properties of the previously presented algorithm of \cite{BalKohNeuPetPet:2022:mrc}. However, there exists a subclass of functions for which the algorithm is convergent. The advantage of this algorithm is its linear complexity (in terms of convexification grid points $N$).

All computations based on these algorithms were performed with the free and open source numerical relaxation code NumericalRelaxation.jl, which is available on GitHub.

\subsubsection{Microstructure reconstruction}
The convexification approach above, in particular the HROC-algorithm \cite{KohNeuPetPetBal:2024:hrs}, has been shown to perform well in relevant macroscopic boundary value problems with concurrent convexification at the integration points. Although not necessary, a graphical illustration of possible microstructures in line with the rank-one convex envelopes might be useful. Then, a representative microscopic deformation field $\varphi$ on a domain $\Omega$ can be determined, based on a fixed oscillatory gradient $\boldsymbol{F}^{\pm}$, which is described by the gradient Young measure corresponding to the lamination.
A detailed description of this procedure is given in \cite[Section 3.3]{KohNeuPetPetBal:2024:hrs}.

It is important to note that the relaxation method does not encode a length scale, but only the ratio of the occurring deformation gradients.
Thus, the width of the laminates or their frequency can be arbitrarily chosen, e.g.~in order to obtain meaningful plots.
However, this choice can lead to mesh-dependent behavior when computing microstructure deformations directly, as shown in e.g.~\cite[Fig. 10]{BarCarHacHop:2004:erm}, \cite[Section 3.1]{KumarEtAl2020}. Therefore, the microstructures presented below should only be seen as one possible realization. Note also that the oscillating microstructure deformation $\varphi$ does not strictly satisfy the homogeneous Dirichlet boundary condition $\varphi(x)=Fx$ on $\partial\Omega$ for finite laminations. Therefore, laminate structures do not necessarily emerge in numerical relaxation approaches based on minimizing the energy directly under boundary conditions, even if $RW(F)=QW(F)$, as shown in Sections~\ref{section:FEM}~and~\ref{section:PINN}.


\begin{figure}[h!]
    \centering
    \includegraphics{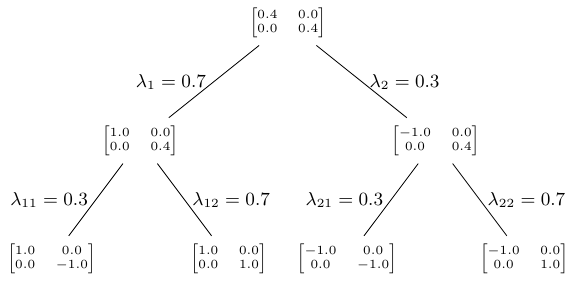}
    \caption{Lamination tree representation of $\mathcal{H}_n$ sequence at uniform compression deformation gradient $0.4\cdot\id_2$.}
    \label{fig:biot-tree-R}
\end{figure}

\subsubsection{Numerical relaxation on $\R^{2\times 2}$}

The energy density $\WBiot$ on the subspace of diagonal matrices in $\R^{2\times 2}$ is shown in Fig.~\ref{fig:biot-hseq-R}.
The relaxation exhibits a zero energy plane in this case, which is spanned in the $F_{11}$--$F_{22}$ plane by the four points $(-1,-1),(-1,1),(1,-1),(1,1)$.
For a given uniform compression deformation gradient $F_0=0.4\cdot\id_2$ the $\mathcal{H}_n$-sequence is shown in Fig.~\ref{fig:biot-hseq-R}.
Note that a permutation of the sequence would yield the same energetic value.
From a lamination perspective, this can be seen as changing the direction of the first and second order laminates in the lamination tree of Fig.~\ref{fig:biot-tree-R}.
Note that this relaxation on $\R^{2\times 2}$ cannot be appropriately represented as a classical deformation gradient field $\grad\varphi$, since some of the laminates (e.g.\ $\diag(-1.0,0.4)$), have negative determinant. This underlines that quasiconvex and rank-one convex relaxation in solid mechanics should always be performed on $\GLp(2)$ or a subset thereof.

For the energy $\Wdist\col\R^{2\times2}$, the algorithm also results in optimal microstructures with self-intersecting laminates, e.g.\ via the rank-one convex combination
\[
	F_0 = \matr{0.4&0\\0&0.4}
	= 0.9\.\matr{0.5&-0.1\\ -0.1&0.5} + 0.1\.\matr{-0.5&0.9\\0.9&-0.5}
	\equalscolon 0.9\.F_1 + 0.1\.F_2
	\,.
\]
Here, the energy of the laminate is given by $0.9\cdot W(F_1)+0.1\cdot W(F_2)=0.68=R_{\R^{2\times2}}W(F_0)=Q_{\R^{2\times2}}W(F_0)$.

\begin{figure}[h!]
    \centering
    \includegraphics{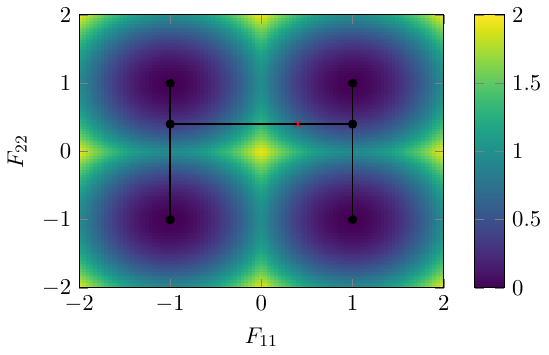}
    \caption{$\mathcal{H}_n$ sequence of Biot-type relaxation on whole $\R^{2\times 2}$ for uniform compression deformation gradient $0.4\cdot\id_2$.
    The red marker indicates the deformation gradient of interest, the black circles show the lamination points and the lines correspond to the rank-one directions.}
    \label{fig:biot-hseq-R}
\end{figure}

\subsubsection{Numerical relaxation on $\GLp(2)$}

The relaxation of $\WBiot$ on $\GLp(2)$ for the uniform compression deformation gradient $F_0=0.4\cdot\id_2$ yields several microstructures with the same optimal energetic value $0.68=R_{\GLp(2)}W(F_0)=Q_{\GLp(2)}W(F_0)$.
Therefore, the laminate tree (or equivalently the $\mathcal{H}_n$ sequence) is highly non-unique.
Fig.~\ref{fig:biot-GL+} shows a small set of possible laminate trees and possible microstructure displacements.
The used algorithm initially stores only a single minimal laminate, thus energetically equivalent laminates are ignored;
the initially found laminate corresponds to the one in Fig.~\ref{fig:tree-biot-GL+2}.

Due to the symmetry properties of the Biot-type energy density, we found qualitatively different laminates by convexifying along a predefined rank-one line.
In this relaxation case, we see qualitatively different possible microstructure fields, shown in
Figs.~\ref{fig:tree-biot-GL+2}, \ref{fig:tree-biot-GL+3}.
We expect that there are even more possible laminates that behave differently but yield the same energy.

\begin{figure*}[h!]
	\newcommand{\tikzsetfigurename}[1]{}
    \centering
    \begin{subfigure}[b]{0.225\textwidth}
        \centering
        \tikzsetfigurename{tree-biot-GL+}
\begin{tikzpicture}
	\pgfmathsetmacro\firstsize{0.9}
	\pgfmathsetmacro\verticalsize{2.0}
	\node (F) at (0,0) {\tiny 
		$\begin{bmatrix} 0.4 & 0.0 \\ 0.0 & 0.4 \end{bmatrix}$};
	\node (F1) at (-\firstsize,-\verticalsize) {\tiny
		$\begin{bmatrix} 0.4 & -0.6 \\ 0.0 & 0.4 \end{bmatrix}$};
	\node (F2) at (\firstsize,-\verticalsize) {\tiny
		$\begin{bmatrix} 0.4 & 0.6 \\ 0.0 & 0.4 \end{bmatrix}$};
	\draw (F) -- (F1) node[midway,left] {\tiny $\lambda_1=0.5$};
	\draw (F) -- (F2) node[midway,right] {\tiny $\lambda_2=0.5$};
\end{tikzpicture}
%
%
        \\[.7em]%
        \includegraphics[width=\textwidth]{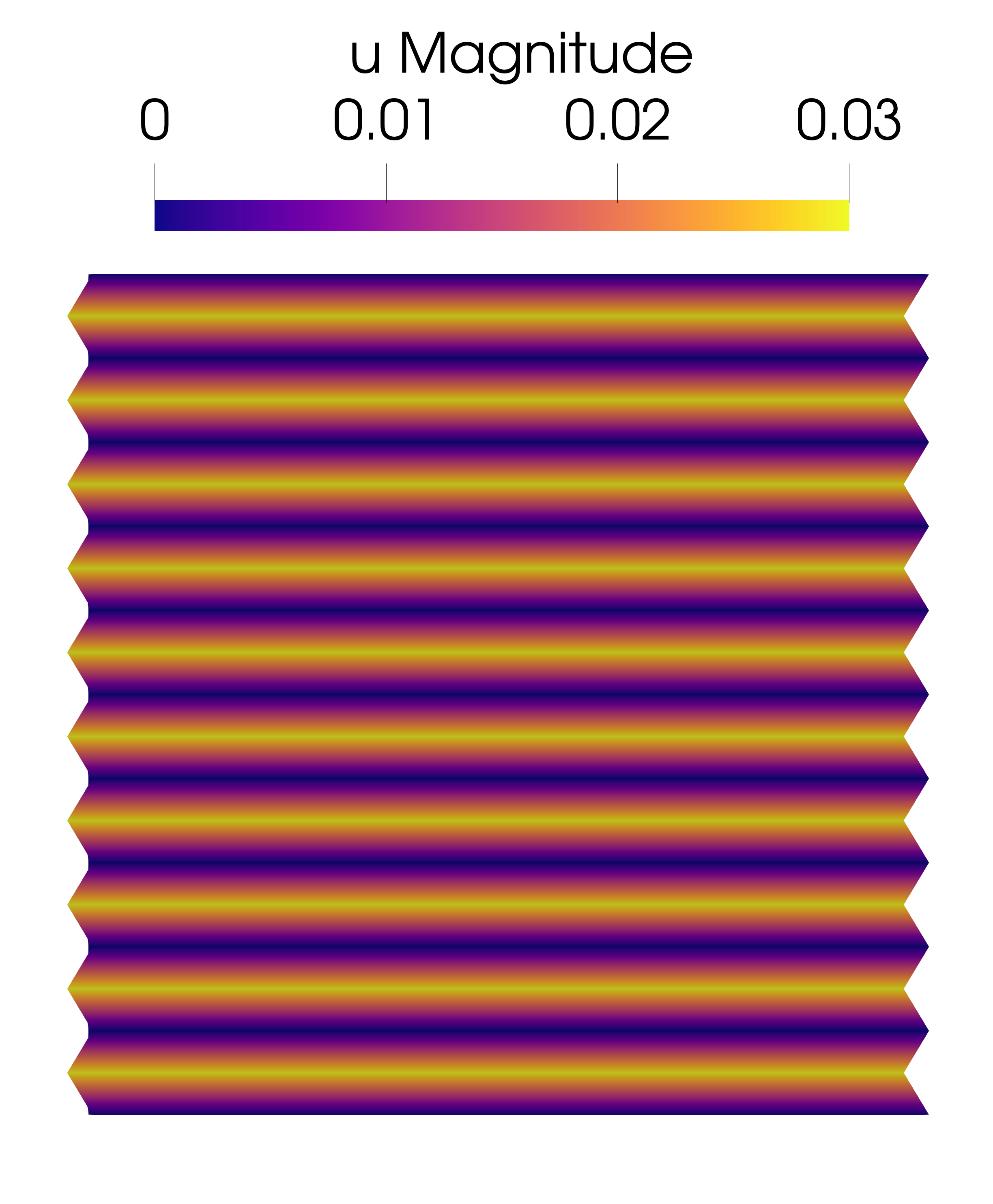}
        \caption{}
        \label{fig:tree-biot-GL+0}
    \end{subfigure}
    \hfill
    \begin{subfigure}[b]{0.225\textwidth}
        \centering
        \tikzsetfigurename{tree1}
\begin{tikzpicture}
	\pgfmathsetmacro\firstsize{0.9}
	\pgfmathsetmacro\verticalsize{2.0}
	\node (F) at (0,0) {\tiny 
		$\begin{bmatrix} 0.4 & 0.0 \\ 0.0 & 0.4 \end{bmatrix}$};
	\node (F1) at (-\firstsize,-\verticalsize) {\tiny
		$\begin{bmatrix} 0.4 & 0.0 \\ -0.6 & 0.4 \end{bmatrix}$};
	\node (F2) at (\firstsize,-\verticalsize) {\tiny
		$\begin{bmatrix} 0.4 & 0.0\\ 0.6 & 0.4 \end{bmatrix}$};
	\draw (F) -- (F1) node[midway,left] {\tiny $\lambda_1=0.5$};
	\draw (F) -- (F2) node[midway,right] {\tiny $\lambda_2=0.5$};
\end{tikzpicture}
%
%
        \\[.7em]%
        \includegraphics[width=\textwidth]{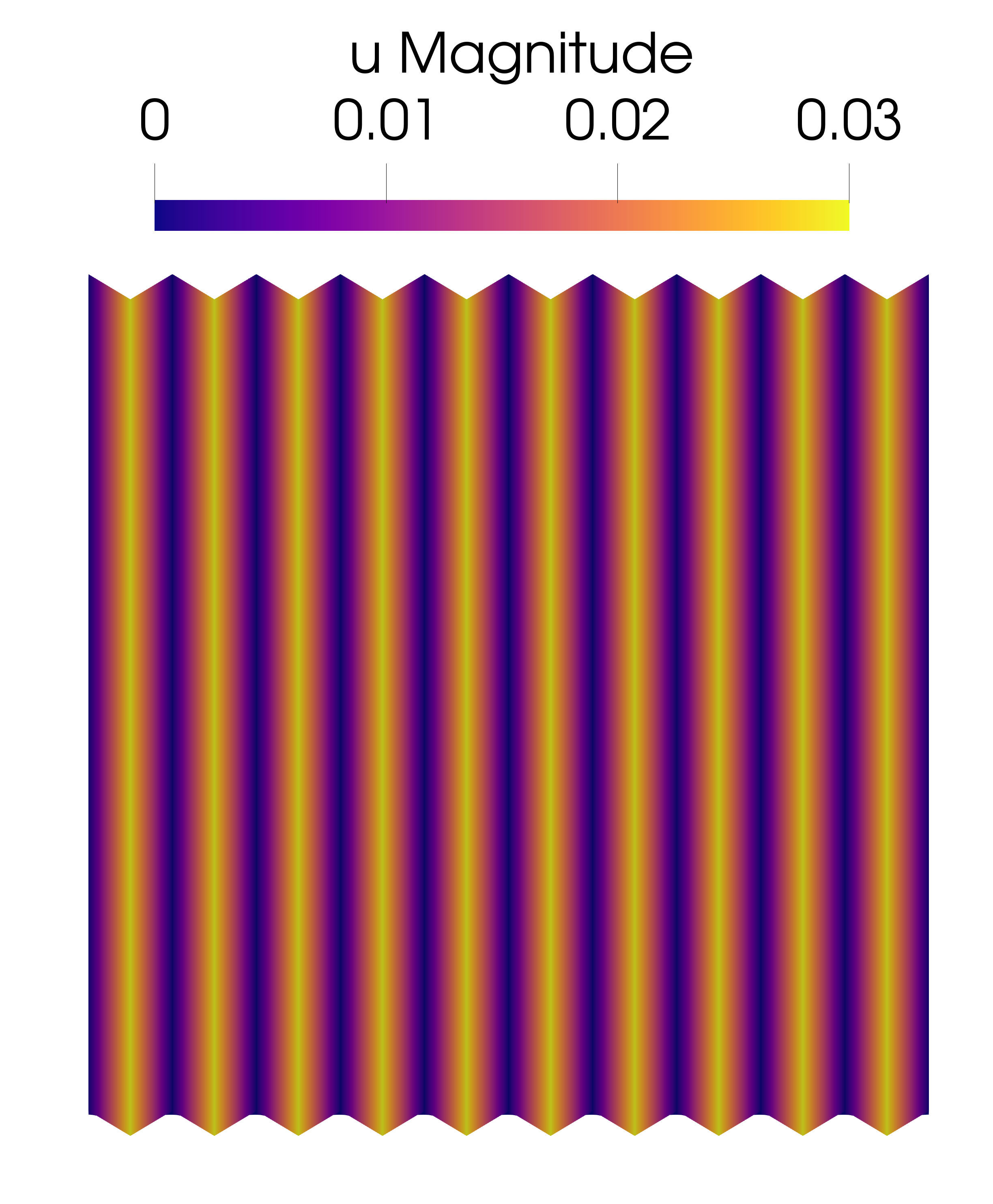}
        \caption{}
        \label{fig:tree-biot-GL+1}
    \end{subfigure}
    \hfill
    \begin{subfigure}[b]{0.225\textwidth}
        \centering
        \tikzsetfigurename{tree-biot-GL+}
\begin{tikzpicture}
	\pgfmathsetmacro\firstsize{0.9}
	\pgfmathsetmacro\verticalsize{2.0}
	\node (F) at (0,0) {\tiny 
		$\begin{bmatrix} 0.4 & 0.0 \\ 0.0 & 0.4 \end{bmatrix}$};
	\node (F1) at (-\firstsize,-\verticalsize) {\tiny
		$\begin{bmatrix} 0.1 & -0.3 \\ 0.3 & 0.7 \end{bmatrix}$};
	\node (F2) at (\firstsize,-\verticalsize) {\tiny
		$\begin{bmatrix} 0.7 & 0.3 \\ -0.3 & 0.1 \end{bmatrix}$};
	\draw (F) -- (F1) node[midway,left] {\tiny $\lambda_1=0.5$};
	\draw (F) -- (F2) node[midway,right] {\tiny $\lambda_2=0.5$};
\end{tikzpicture}
%
%
        \\[.7em]%
        \includegraphics[width=\textwidth]{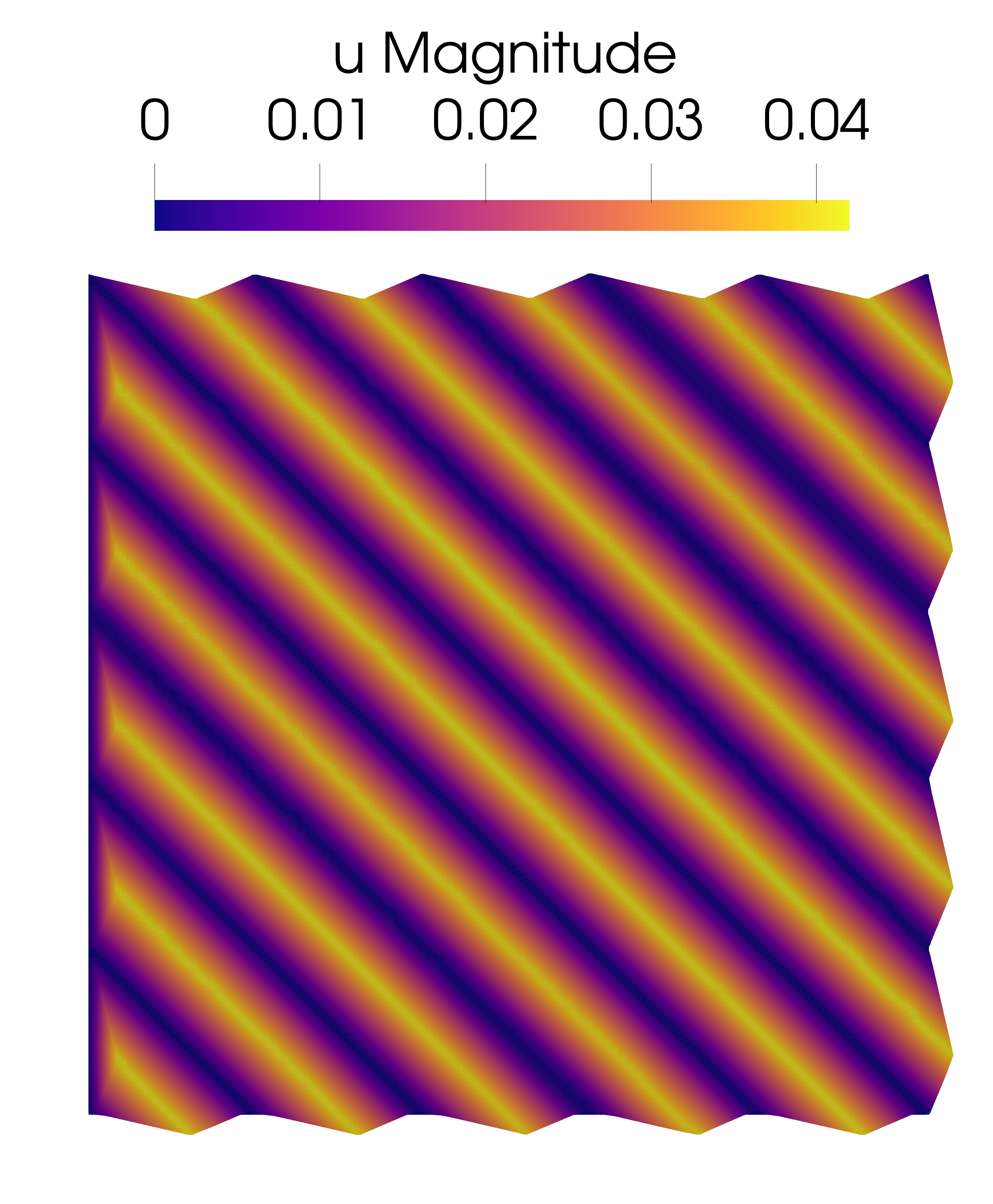}
        \caption{}
        \label{fig:tree-biot-GL+2}
    \end{subfigure}
    \hfill
    \begin{subfigure}[b]{0.225\textwidth}
        \centering
     	\tikzsetfigurename{tree-biot-GL+}
\begin{tikzpicture}
	\pgfmathsetmacro\firstsize{0.9}
	\pgfmathsetmacro\verticalsize{2.0}
	\node (F) at (0,0) {\tiny 
		$\begin{bmatrix} 0.4 & 0.0 \\ 0.0 & 0.4 \end{bmatrix}$};
	\node (F1) at (-\firstsize,-\verticalsize) {\tiny
		$\begin{bmatrix} 0.1 & 0.3 \\ -0.3 & 0.7 \end{bmatrix}$};
	\node (F2) at (\firstsize,-\verticalsize) {\tiny
		$\begin{bmatrix} 0.7 & -0.3 \\ 0.3 & 0.1 \end{bmatrix}$};
	\draw (F) -- (F1) node[midway,left] {\tiny $\lambda_1=0.5$};
	\draw (F) -- (F2) node[midway,right] {\tiny $\lambda_2=0.5$};
\end{tikzpicture}
%
%
     	\\[.7em]%
     	\includegraphics[width=\textwidth]{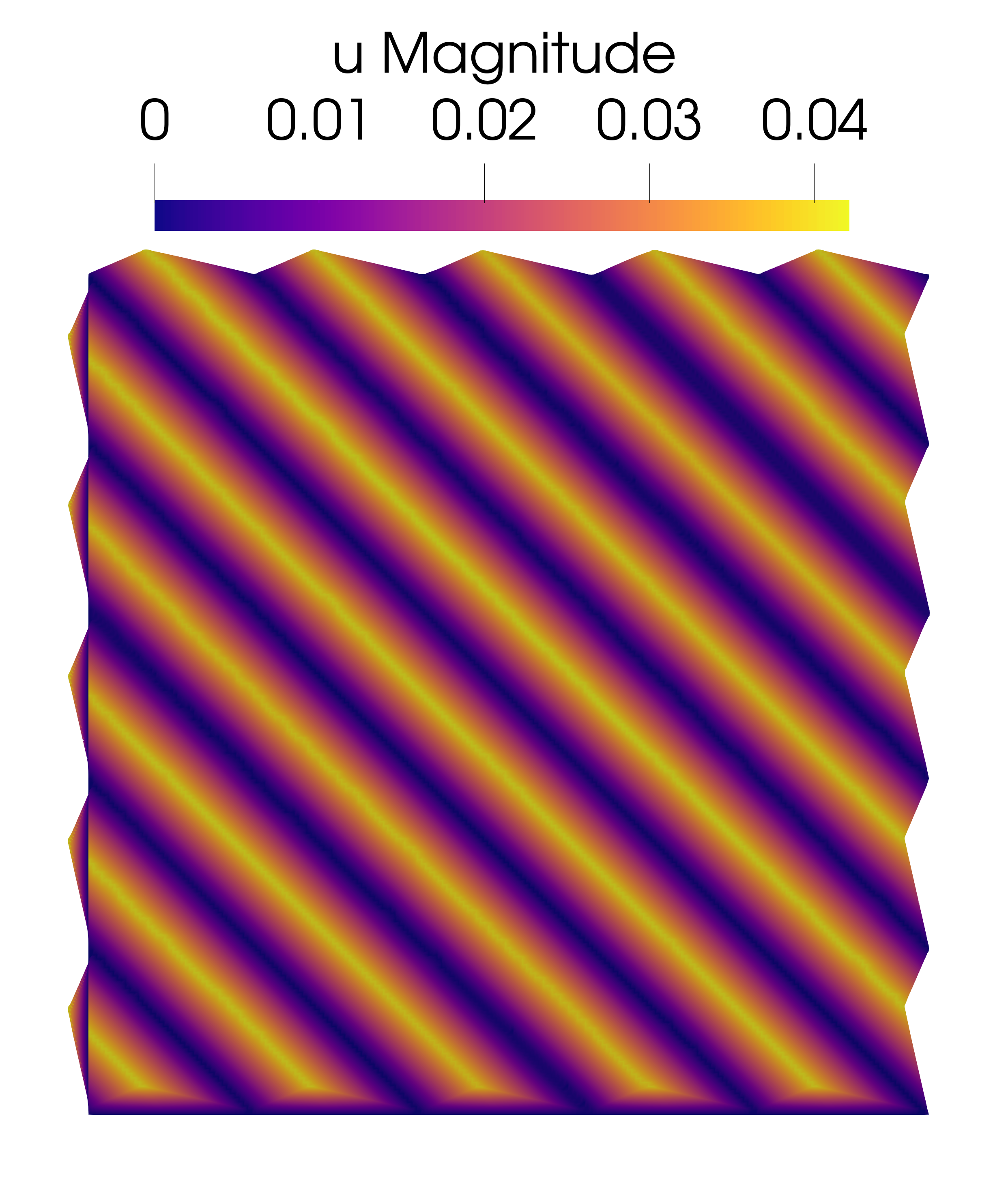}
        \caption{}
        \label{fig:tree-biot-GL+3}
    \end{subfigure}
    \caption{\small Four different laminates, all yielding the same energy $0.68=RW(F_0)=QW(F_0)$, and associated possible microstructure displacement fields for uniform compression deformation gradient $0.4 \cdot \id_2$.
    For all microstructure displacement fields the frequency is set to $10$.%
    }
    \label{fig:biot-GL+}
\end{figure*}
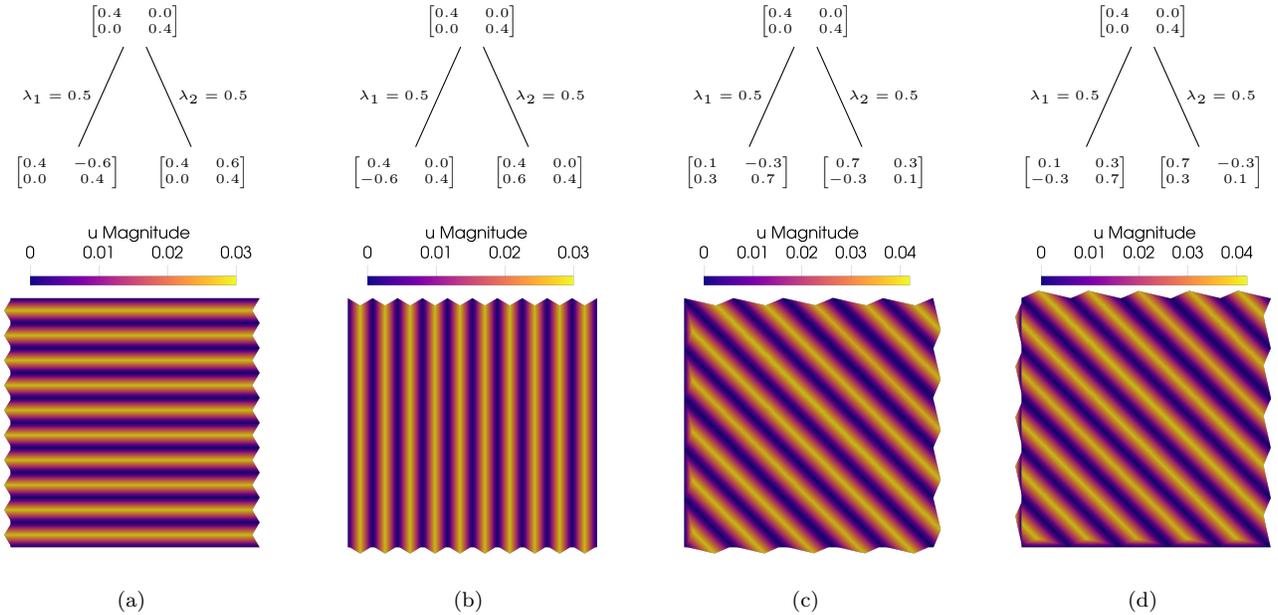

\let\boldsymbol\boldsymbolold%

\subsection{Comparison of numerical relaxation methods}

Table~\ref{table:numerical_results} shows the comparison between the numerically obtained approximations of the quasiconvex and rank-one convex envelopes of the Biot-type functions. Note that for the numerical simulations, the approximation of the quasiconvex envelope is obtained via
\[
	QW(F_0) \approx \frac{1}{\meas{\Omega}}\.\int_\Omega W(\grad\varphi(x))\,\dx
	\,,
\]
where $\varphi$ denotes the numerically obtained deformation, and that $\Omega=[-1,1]^2$ in both the FEM and the PINN simulations. Furthermore, in all cases considered here, the rank-one convex relaxation is already quasiconvex,\footnote{Of course, any $W\col\R^{2\times2}\to\R$ with $QW\neq RW$ would represent a counterexample to Morrey's conjecture \cite{morrey1952quasi,agn_voss2021morrey,agn_voss2021morrey2}.} i.e.\ $QW=RW$.
Therefore,
the ROC algorithm converges to the quasiconvex envelope in the chosen examples, as expected, due to its convergence properties with respect to the rank-one convex envelope approximation. The HROC algorithm also converges, as the chosen example fulfills the necessary properties to make HROC a convergent algorithm.

However, note that -- since $QW\neq RW$ in general -- any algorithm for computing the rank-one envelope directly can only provide an upper bound on $QW$ in the general case. In particular, if $W$ is already rank-one convex, then any such method is unsuitable for determining whether $W$ is also quasiconvex, whereas the methods shown in Sections~\ref{section:FEM}~and~\ref{section:PINN} might find a deformation satisfying homogeneous Dirichlet boundary conditions with an energy value below the one corresponding to the homogenoeus deformation and thus demonstrate the non-quasiconvexity.

\begin{table}[h!]
	\centering
	\renewcommand{\arraystretch}{1.47}
	\begin{tabular}{l|cccccc}%
        & FEM & PINN & ROC & HROC & $QW(F_0)=RW(F_0)$ & $W(F_0)$%
		\\\hline%
		$W=\WBiot$ on $\R^{2\times2}$&
            0.377 & 0.353 & 0 & 0 & 0 & 0.72
		\\
		$W=\Wdist$ on $\R^{2\times2}$&
            0.681 & 0.681 & 0.68 & 0.68 & 0.68 & 0.72
		\\
		$W=\WBiot$ on $\GLp(2)$&
            0.681 & 0.681 & 0.68 & 0.68 & 0.68 & 0.72
	\end{tabular}%
	\caption{\label{table:numerical_results}Comparison of numerical relaxation results for $F_0=0.4\cdot\id_2$.}
\end{table}

In all cases, the difference between the constrained (or penalized) relaxation on $\GLp(2)$ and the unconstrained relaxation on $\R^{2\times2}$ of the energy $\WBiot$ is indeed significant. However, the resulting energy levels for $\Wdist$ are independent of the determinant constraint; note again that
\[
	Q_{\GLp(2)}\Wdist(F) = Q_{\R^{2\times2}}\Wdist(F)
\]
for all $F\in\GLp(2)$ according to \eqref{eq:summary_envelope_constrained}.

While neither the FEM nor the PINN method approximates the analytical relaxation for $\WBiot$ in the unconstrained case, it should be noted that the computed energies were significantly below that of the homogeneous deformation in each case.
In particular, PINNs can be a viable approach to find violations of quasiconvexity for elastic energy potentials. In order to precisely determine the value of the quasiconvex envelope, however, classical methods such as trust-region FEM calculations or (if it is known that $RW=QW$) rank-one sequential lamination, should generally be considered more reliable. Furthermore, for practical applications, it may already be sufficient to check for rank-one convexity instead of quasiconvexity, in which case rank-one lamination using the ROC or HROC algorithms is reasonable.

\section{Conclusion}
\label{section:conclusion}
 
In nonlinear elasticity theory, strictly observing the determinant constraint $\det F>0$ on the deformation gradient $F$ is important not only when solving classical boundary value problems, but also for determining the quasiconvex envelope of energy potentials. For the planar Biot-type energy $\WBiot(F)=\norm{\sqrt{F^TF}-\id_2}^2$, the relaxation over all orientation-preserving deformations indeed differs from the unconstrained relaxation, as shown in Section~\ref{section:analyticalResults}.

Numerical techniques for computing the relaxation of non-quasiconvex energies need to take the determinant constraint into account as well. As demonstrated in Section~\ref{sec:numerical_results}, numerical methods based on rank-one convexification, finite elements or physics-informed neural networks all require the condition $\det F >0$ to be enforced explicitly either a priori or via suitable penalty terms in order to ensure that the computed energy levels approximate the correct (i.e.\ constrained) envelope.

In case of the Biot energy, however, due to the relation
\[
	Q_{\GLp(2)}\WBiot(F) = Q_{\GLp(2)}\Wdist(F) = Q_{\R^{2\times2}}\Wdist(F)
	\,,
\]
the constrained relaxation of $\WBiot$ is indeed equivalent to the unconstrained relaxation of the related energy $\Wdist$. Furthermore, as summarized in Section~\ref{section:analytical_results_summary}, the quasiconvex envelopes of $\WBiot$ and $\Wdist$ can be computed explicitly, both in the constrained and in the unconstrained case. In addition to these specific energies, our results include a general formula for the quasiconvex relaxation of Valanis-Landel-type energy functions on $\Rnn$.

\bigskip\par

\noindent{\bf Acknowledgements:}  This research of P.~Neff has been funded by the Deutsche Forschungsgemeinschaft (DFG, German Research Foundation), project ID 415894848, reference ID NE902/8-1 (P.~Neff)). We also acknowledge the DFG for funding within the Priority Program 2256 ("Variational Methods for Predicting Complex Phenomena in Engineering Structures and Materials"), project ID 441154176, reference ID BA2823/17-1.
Furthermore, the free and open source software community, especially of the Julia programming language and Ferrite.jl, as well as scientific discussions with Dennis Ogiermann, Timo Neumeier, Malte~A.~Peter, and Daniel Peterseim are gratefully acknowledged.

\medskip






\section*{References}

\printbibliography[heading=none]

\appendix

\begin{footnotesize}
	\section{Auxiliary results}
	We recall first some results regarding the convexity of some functions defined on spaces of matrices \cite{davis1957}.
	\begin{theorem}\label{Davis1}
		A unitary-invariant function (i.e. $f(U^{-1}X U)=f(X)$ for any $X\in \R^{n\times n}$ and any $n\times n$ unitary $U$) from $n\times n$ hermitian matrices to a partly-ordered real vector space is convex if and only if the corresponding symmetric function of $n$ real variables is convex.
	\end{theorem}
Theorem~\ref{Davis1} also holds for orthogonally invariant functions of real symmetric matrices, as well as for functions defined only for matrices whose spectra are restricted to a given finite or infinite interval.

Pipkin \cite{pipkin1994relaxed} first gave a formula for computing the quasiconvex envelope of stored energy functions defined on the set of real $3\times 2$ matrices that are left $\OO(3)$-invariant and convex with respect to the strain tensor. This formula was later extended to a more general case:
Using a different method than Pipkin's proof, Raoult \cite{raoult2010quasiconvex} showed the following lemma.
\begin{lemma}[\cite{raoult2010quasiconvex,pipkin1994relaxed}]\label{lemmaRP}
	Let $m\geq n$ and $Y:\R^{n\times m}\to \R$ be a left $\OO(n)$-invariant, rank-one convex mapping. Then the mapping $\widetilde{Y}\col\Symp(m)\to \R$ such that $Y(F)=\widetilde{Y}(F^TF)$ for all $F\in \R^{n\times m}$ satisfies 
	\begin{align}
	\widetilde{Y}(C)\geq \widetilde{Y}(C+S) \qquad \text{for all }\;C,S\in \Symp(m).
	\end{align}
\end{lemma}
\noindent Here and throughout, we denote by $\Symp(m)$ the set of positive definite symmetric $m\times m$-matrices. Using this Lemma~\ref{lemmaRP}, Raoult \cite{raoult2010quasiconvex} then proved the following theorem.
	\begin{theorem}[Generalized Pipkin's formula]\label{Raoult}
		Let $W:F\in\R^{n\times m}\mapsto \R$, $m\geq n$, be a left $\OO(n)$-invariant, bounded from below stored energy function such that the associated function $\widetilde{W}\colon C\mapsto \widetilde{W}(C)$ is convex on $\Symp(m)$. Then
		\begin{align}
		Q W(F)=\inf\{S\in \Symp(m) \setvert W(F^T F+S)\}\,.
		\end{align}
	\end{theorem}
\noindent It was also proven by Raoult that the Pipkin's formula fails for $m>n$.

\section{An alternative method for finding \boldmath$Q\WBiot$ on \boldmath$\R^{2\times2}$}
\label{appendix:alternativeProof}

\begin{lemma}
\label{lemma:schur}
	Let $C,S\in\PSymn$ such that $C$ is a diagonal matrix. Then
	\[
		\norm{\sqrt{C+S}-\id_n} \geq \norm{\sqrt{C+\diag(S)}-\id_n}\,,
	\]
	where $\diag(S) = \diag(S_{11},\dotsc,S_{nn})$.
\end{lemma}
\begin{proof}
	Since $C+S\in\PSymn$ and $C+\diag(S)=\diag(C+S)$, it is sufficient to show that
	\[
		\norm{\sqrt{P}-\id_n}^2 \geq \norm{\sqrt{\diag(P)}-\id_n}^2
	\]
	for all $P\in\PSymn$. To this end, let $\lambda_1,\dotsc,\lambda_n>0$ denote the eigenvalues of $P$, and let $a_1,\dotsc,a_n>0$ denote the diagonal entries of $P$. Then due to a Theorem of Schur \cite{schur1923uber} (cf.~\cite[Theorem B.1]{marshall2010}),
	\[
		a \colonequals (a_1,\dotsc,a_n) \prec (\lambda_1,\dotsc,\lambda_n) \equalscolon \lambda
		\,,
	\]
	i.e.\ $a$ is majorized by $\lambda$. Furthermore, since the mapping $t\mapsto(\sqrt{t}-1)^2$ is convex on $[0,\infty)$, the function
	\[
		g\col[0,\infty)^n\to\R\,,\quad
		g(x_1,\dotsc,x_n) = \sum_{k=1}^n (\sqrt{x_k}-1)^2
	\]
	is Schur-convex \cite{schur1923uber} (cf.~\cite[Proposition C.1]{marshall2010}), which immediately implies that
	\[
		\norm{\sqrt{P}-\id_n}^2 = g(\lambda) \geq g(a) = \norm{\sqrt{\diag(P)}-\id_n}^2
		\,.
		\qedhere
	\]
\end{proof}

\begin{lemma}
	The mapping $Z\colon\Symp(2)\to \R$ given by
	\begin{align}
		Z(C)=\norm{\sqrt{C}-\id_2}^2\,.
	\end{align}
	is convex on $\Symp(2)$.
\end{lemma}
\begin{proof}
	According to Theorem~\ref{Davis1}, $Z$ is convex if and only if the function $w\colon\Rp\to \R$ which represents $Z$ in terms of singular values, i.e.
	\begin{align}
		{Z}({C})&=\norm{\sqrt{{C}}-\id_2}^2=\sum_{i=1}^2(\sqrt{\mu_i}-1)^2=w(\mu_1,\mu_2)\,,
	\end{align}
	where $\mu_1,\mu_2$ are the eigenvalues of ${C}$, is convex in $(\mu_1,\mu_2)$.
	In order to verify the convexity of $w$, we compute its Hessian matrix on $\R_+^2$,
	\begin{align}
	\D^2_{(\mu_1,\mu_2)}w=\left(
	\begin{array}{cc}
	\frac{ 1}{8\, \mu _1^{3/2} } & 0\\
	0 &\frac{ 1}{8\, \mu _2^{3/2} } 
	\end{array}
	\right)\in \Symp(2)
	\,,
	\end{align}
	and thereby directly observe that $w$ is convex everywhere on $\R_+^2$.
\end{proof}
\begin{proposition}\label{quasiWBiot}
	The quasiconvex envelope of
	\begin{align*}
		&\WBiot\colon\R^{2\times2}\to\R\,,\qquad
		\WBiot(F)=\norm{\sqrt{F^TF}-\id_2}^2
	\intertext{on $\Rnn$ is given by}
		Q&\WBiot(F) = \sum_{k=1}^n[\lambda_k-1]^2_+
		\,,\qquad\qquad 
		[x]_+\colonequals
		\begin{cases}
			x &\casesif x\geq 0\,,\\
			0 &\casesif x<0\,,
		\end{cases}
	\end{align*}
	for all $F\in\Rnn$ with singular values $\lambda_1,\dotsc,\lambda_n\geq0$.
\end{proposition}
\begin{proof}
	We can express $\WBiot$ as $\WBiot(F)=Z({C})$, where ${C}=F^TF\in \Symp(2)$ with
	$
	{Z}({C})=\norm{\sqrt{{C}}-\id_2}^2\,.
	$
	Clearly, ${Z}$ is convex as a function of ${C}$ on $\Symp(2)$. By Pipkin's formula \cite{pipkin1993convexity} (cf.~Theorem~\ref{Raoult}),
	\begin{align}
	QZ(F)=\inf_{S\in \Symp(2)}{Z}({C}+S)\,,\qquad \qquad {C}=F^TF\in \Symp(2)\,.
	\end{align}
	Hence, to compute the quasiconvex envelope of $\WBiot$, we need to compute $\inf_{S\in \Symp(2)}{Z}({C}+S)$. The calculations are similar to those given in \cite{le1995quasiconvex} where the quasiconvex envelope of the St.\ Venant-Kirchhoff energy is computed:
	Let us fix a $C\in \Symp(2)$ and define
	\[
		J_C\col\Symp(2)\to \R\,,\quad J_C(S)=Z(C+S)\,.
	\]
	Due to the convexity of $C\mapsto Z(C)$, the function $S\mapsto J_C(S)$ is convex as well. In addition, since
	\begin{align}
	\lVert \sqrt{C}\rVert-2\leq	\lVert \sqrt{C}-\id_2\rVert
	\end{align}
	and
	\begin{align}
	\lVert {C}\rVert=[\lambda_1^2(C) +\lambda_2^2(C)]^{\frac{1}{2}}\leq [\lambda_1(C) +\lambda_2(C)]=\lVert \sqrt{C}\rVert^2\,,
	\end{align}
	we deduce that
	\begin{align}
	\lVert {C}\rVert^{\frac{1}{2}}-2\leq	\lVert \sqrt{C}-\id_2\rVert,
	\end{align}
	i.e. that $C\mapsto Z(C)$ is coercive on $\Symp(2)$. The same property (coercivity) will be transmitted to $S\mapsto J_C(S)$. 
	
	Since $C$ is fixed for now,
	\begin{align}
	\norm{\sqrt{C+S}-\id_2}^2=\norm{Q^T(\sqrt{C+S}-\id_2)Q}^2=
	\norm{\sqrt{Q^T(C+S)Q}-Q^TQ}^2=\norm{\sqrt{Q^TCQ+Q^TSQ}-\id_2}^2\,,
	\end{align}
	and since $Q^T S Q\in \Symp(2)$ for all $S\in \Symp(2)$ and $Q\in O(2)$, we may assume that $C$ is a diagonal matrix of the form
	\begin{align}
	C=\diag(c_{11},c_{22})\,.
	\end{align}
	Lemma~\ref{lemma:schur} immediately implies that
	\begin{align}
	\inf_{S\in \Symp(2)}J_C(S)\geq	\inf_{(s_{11},s_{22})\in \R_+^2}J_C((\diag(s_{11},s_{22})))\,,
	\end{align}
	i.e.\ that minimizing $J_C(S)$ among semi-positive definite matrices is equivalent to minimizing $J_C(S)$ among diagonal semi-positive definite matrices.
	Therefore, we have to minimize the function $j\colon\R_+^2\to \R$ with
	\begin{align}
	j_C(s_{11},s_{22})&=J_C(\diag(s_{11},s_{22}))\notag\\&=(	\sqrt{\lambda_1(\diag(c_{11}+s_{11},c_{22}+s_{22}))}-1)^2+(		\sqrt{\lambda_2(c_{11}+s_{11},c_{22}+s_{22})}-1)^2]\notag\\
	\\&=(	\sqrt{c_{11}+s_{11}}-1)^2+(	\sqrt{c_{22}+s_{22}}-1)^2\notag
	\end{align}
	on $\R_+^2$.
	We thus have to find $s_{11},s_{22}\in \R_+$ which satisfy the optimality conditions (minimization over nonnegative orthant \cite[Page 173]{beck2014introduction})
	\begin{align}
	\langle	\D j_C(s_{11},s_{22}),(t_{11},t_{22})^T\rangle&\geq 0 \quad \text{for all} \quad (t_{11},t_{22})^T\in \R_+^2,
\notag	\\
	\langle	\D j_C(s_{11},s_{22}),(s_{11},s_{22})^T\rangle&= 0\,.
	\end{align}
	These optimality conditions are equivalent to
	\begin{alignat}{2}
	&\frac{\partial j_C}{\partial s_{11}}(s_{11},s_{22})\geq 0, \qquad\quad &&\frac{\partial j_C}{\partial s_{22}}(s_{11},s_{22})\geq 0,\notag\\
	&\frac{\partial j_C}{\partial s_{11}}(s_{11},s_{22})\, s_{11}= 0, \qquad &&\frac{\partial j_C}{\partial s_{22}}(s_{11},s_{22})\, s_{22}= 0.
	\end{alignat}
	Since
	\begin{align}
	\D j_C(s_{11},s_{22})=\left(
	\displaystyle\begin{array}{c}
	\frac{\sqrt{c_{11}+s_{11}}-1}{\sqrt{c_{11}+s_{11}}} \vspace{1.5mm}\\
	\frac{\sqrt{c_{11}+s_{11}}-1}{\sqrt{c_{11}+s_{11}}} \\
	\end{array}
	\right)\,,
	\end{align}
	these optimality conditions read
	\begin{alignat}{2}
	&1-\frac{1}{\sqrt{c_{11}+s_{11}}}\geq 0,\qquad \quad &&1-\frac{1}{\sqrt{c_{22}+s_{22}}}\geq 0,\notag\\ \Bigl(&1-\frac{1}{\sqrt{c_{11}+s_{11}}}\Bigr)s_{11}= 0, 
	\quad &\Bigl(&1-\frac{1}{\sqrt{c_{11}+s_{22}}}\Bigr)s_{22}= 0\,.
	\end{alignat}
	We distinguish three cases:
	\begin{itemize}
		\item If $c_{11}<1$ and $c_{22}<1$, since $s_{11}=0$ and $s_{22}=0$ do not satisfy the first two optimality condition,
		\begin{align}s_{11}=1-c_{11},\qquad s_{22}=1-c_{22}\,,
		\end{align}
		which satisfies the first two optimality conditions as well, and
		\begin{align}\min_{(s_{11},s_{22})\in \R^2_+}
		j_C(s_{11},s_{22})=0=[(	\sqrt{c_{11}}-1)^2]_++[(	\sqrt{c_{22}}-1)^2]_+\,.
		\end{align}
		\item If $c_{11}>1$ and $c_{22}>1$, then the first two optimality conditions are strictly satisfied and we have to choose 
		\begin{align} s_{11}=0,\quad s_{22}=0
		\end{align}
		and
		\begin{align}\min_{(s_{11},s_{22})\in \R^2_+}
		j_C(s_{11},s_{22})&=(	\sqrt{c_{11}}-1)^2+(	\sqrt{c_{22}}-1)^2=\norm{\sqrt{{C}}-\id_2}^2\notag\\&=[(	\sqrt{c_{11}}-1)^2]_++[(	\sqrt{c_{22}}-1)^2]_+\,.
		\end{align}
		\item If $c_{11}>1$ and $c_{22}<1$, then from the same reasons as above, we have to choose 
		\begin{align} s_{11}=0\quad \text{and}\quad s_{22}=1-c_{22}
		\end{align}
		and
		\begin{align}\min_{(s_{11},s_{22})\in \R^2_+}
		j_C(s_{11},s_{22})&=(	\sqrt{c_{11}}-1)^2+(	\sqrt{0}-1)^2=\norm{\sqrt{{C}}-\id_2}^2\notag\\
		&=[(	\sqrt{c_{11}}-1)^2]_++[(	\sqrt{c_{22}}-1)^2]_+\,.
		\end{align}
		
		\item If $c_{11}<1$ and $c_{22}>1$, then we have to choose 
		\begin{align} s_{11}=1-c_{11}\quad \text{and}\quad s_{22}=0
		\end{align}
		and
		\begin{align}\min_{(s_{11},s_{22})\in \R^2_+}
		j_C(s_{11},s_{22})=\Big[(	\sqrt{0}-1)^2+(	\sqrt{c_{22}}-1)^2\Big]=[(	\sqrt{c_{11}}-1)^2]_++[(	\sqrt{c_{22}}-1)^2]_+\,.
		\end{align}
	\end{itemize}
	Replacing $c_{11}$ and $c_{22}$ by the eigenvalues of $C$, $0 \leq \lambda_1(C)\leq \lambda_2(C)$, i.e.~by the singular values $0\leq \vartheta_1(F)\leq \vartheta_2(F)$ of $F$, completes the proof.
\end{proof}

\begin{remark}
	The extension of this relaxation result to the Biot energy energy $\mu\,\norm{\sqrt{F^T F}-\id_3}+\frac{\lambda}{2}\,[\tr(\sqrt{F^T F}-\id_3)]
^2$ is not possible with the help of the Pipkin formula \cite{pipkin1994relaxed}, since the trace term is not a convex function of $C$.\end{remark}

\end{footnotesize}

\end{document}